\documentclass[dvipdfmx]{article}
\usepackage{amsmath,amsthm,amssymb}
\usepackage{mathrsfs,bm}
\usepackage{enumitem, appendix}

\usepackage{tikz}

\usepackage[dvipdfmx]{hyperref}
\hypersetup{
    colorlinks=true,
    linkcolor=blue,
}

\newtheorem{theo}{Theorem}[section]
\newtheorem{prop}[theo]{Proposition}
\newtheorem{coro}[theo]{Corollary}
\newtheorem{lemm}[theo]{Lemma}

\theoremstyle{definition}
\newtheorem{defi}[theo]{Definition}
\newtheorem*{note*}{Note}
\newtheorem*{claim*}{Claim}
\newtheorem*{exam*}{Example}
\newtheorem*{rema*}{Remark}
\newtheorem{assu}[theo]{Assumption}

\author{Shota FUKUSHIMA\thanks{Graduate School of Mathematical Sciences, the University of Tokyo, 3-8-1 Komaba, Meguro-ku, Tokyo, 153-8914, Japan. 
Email: fukusima@ms.u-tokyo.ac.jp} 
\thanks{The author is supported by Leading Graduate Course for Frontiers of Mathematical Sciences and Physics (FMSP), at Graduate School of Mathematical Science, the University of Tokyo. }}
\title{Propagation of singularities under Schr\"odinger equations on manifolds with ends}

\newcommand{\vol}{\mathrm{vol}}
\newcommand{\jbracket}[1]{\left\langle {#1} \right\rangle}
\newcommand{\pprime}{{\prime\prime}}

\newcommand{\rmop}[1]{\mathop{\mathrm{#1}}}

\newcommand{\grad}{\rmop{grad}}

\newcommand{\diff}{\mathrm{d}}

\newcommand{\WFrh}{\rmop{WF}\nolimits^\mathrm{rh}}

\newcommand{\Oph}{\mathop{\mathrm{Op}}\nolimits_\hbar}
\newcommand{\Ophloc}[1][\iota]{\Oph^{\mathrm{loc}, {#1}}}

\newcounter{step}
\newcommand{\initstep}{\setcounter{step}{0}}
\newcommand{\step}{
    \par\vspace{2mm}
    \refstepcounter{step}
    \noindent\textbf{Step {\thestep}.}
}

\newenvironment{divstep}{\setcounter{step}{0}}{\par\vspace{2mm}}

\newcommand{\fstep}[1]{
    \par\vspace{2mm}
    \noindent\textbf{#1.}
}

\newcounter{enumpartcounter}

\numberwithin{equation}{section}

\begin{document}
\maketitle
\begin{abstract}
    We prove a microlocal smoothing effect of Schr\"odinger equations on manifolds. We employ radially homogeneous wavefront sets introduced by Ito and Nakamura (Amer. J. Math., 2009). In terms of radially homogeneous wavefront sets, we can apply our theory to both of asymptotically conical and hyperbolic manifolds. We relate wavefront sets in initial states to radially homogeneous wavefront sets in states after a time development. We also prove a relation between radially homogeneous wavefront sets and homogeneous wavefront sets and prove a special case of Nakamura (2005). 
\end{abstract}

\section{Introduction}

\subsection{Motivation: homogeneous wavefront sets on Euclidean spaces}\label{subs_hwf_euclid}

For proving a microlocal smoothing effect of a Schr\"odinger equation 
\[
    i\frac{\partial u}{\partial t}(t, x)=Hu(t, x) \quad (t, x)\in \mathbb{R}\times \mathbb{R}^n, 
\]
it is known that one does not need only the usual wavefront sets, which is localized with respect to positions, but also another notion of a wavefront set by which we can access to behavior of functions near infinity. One of the methods is a homogeneous wavefront set. Recall the definition of (homogeneous) wavefront sets from \cite{Nakamura05}: 

\begin{divstep}
\fstep{Wavefront sets on Euclidean spaces}For $u\in L^2(\mathbb{R}^n)$, we define a set $\rmop{WF}(u)\subset T^*\mathbb{R}^n\setminus 0$ by the following property: a point $(x_0, \xi_0)\in T^*\mathbb{R}^n\setminus 0$ \textit{does not} belong to $\rmop{WF} (u)$ if there exists $a\in C_c^\infty (T^*\mathbb{R}^n)$ such that $a=1$ near $(x_0, \xi_0)$ and 
    \[
        \|a^\mathrm{w}(x, \hbar D)u\|_{L^2}=O(\hbar^\infty)
    \]
    as $\hbar \to 0$. 
    
    Here $a^\mathrm{w}(x, \hbar D)$ is the usual semiclassical Weyl quantization of the symbol $a(x, \xi)$: 
    \[
        a^\mathrm{w}(x, \hbar D)u(x):=\frac{1}{(2\pi\hbar)^n}\int_{\mathbb{R}^{2n}} a\left(\frac{x+y}{2}, \xi\right) e^{i\xi \cdot (x-y)/\hbar}u(y)\, \diff y \diff\xi. 
    \]

\fstep{Homogeneous wavefront sets on Euclidean spaces}For $u\in L^2(\mathbb{R}^n)$, we define a set $\rmop{HWF}(u)\subset T^*\mathbb{R}^n\setminus \{(0, 0)\}$ by the following property: a point $(x_0, \xi_0)\in T^*\mathbb{R}^n\setminus \{(0, 0)\}$ \textit{does not} belong to $\rmop{HWF} (u)$ if there exists $a\in C_c^\infty (T^*\mathbb{R}^n)$ such that $a=1$ near $(x_0, \xi_0)$ and 
    \[
        \|a^\mathrm{w}(\hbar x, \hbar D)u\|_{L^2}=O(\hbar^\infty)
    \]
    as $\hbar \to 0$. 

    Here $a^\mathrm{w}(\hbar x, \hbar D)$ is the semiclassical Weyl quantization of the symbol $a(\hbar x, \xi)$: 
\[
    a^\mathrm{w}(\hbar x, \hbar D)u(x):=\frac{1}{(2\pi\hbar)^n}\int_{\mathbb{R}^{2n}} a\left(\frac{\hbar x+\hbar y}{2}, \xi\right) e^{i\xi \cdot (x-y)/\hbar}u(y)\, \diff y\diff \xi. 
\]
\end{divstep}

Nakamura \cite{Nakamura05} proved that, if 
\begin{itemize}
    \item $H=-\sum_{j, k=1}^n \partial_{x_j}a_{jk}(x)\partial_{x_k}/2+V(x)$, with a positive definite matrix $(a_{jk}(x))_{j, k=1}^n$, $a_{jk}(x), V(x)\in \mathbb{R}$, $|\partial_x^\alpha (a_{jk}(x)-\delta_{jk})|\leq C_\alpha \jbracket{x}^{-\mu-|\alpha|}$ and $|\partial^\alpha V(x)|\leq C_\alpha \jbracket{x}^{\nu-|\alpha|}$ for some $\mu>0$ and $\nu <2$, and  
    \item a classical orbit $(x(t), \xi(t))$ with respect to the classical Hamiltonian $h_0(x, \xi):=\sum_{j, k=1}^n a_{jk}(x)\xi_j \xi_k/2$ is nontrapping ($|x(t)|\to\infty$ as $t\to\infty$) and has an initial point $(x(0), \xi(0))=(x_0, \xi_0)$ and an asymptotic momentum $\xi_\infty:=\lim_{t\to \infty}\xi(t)$, 
\end{itemize}
then $(x_0, \xi_0)\in \rmop{WF}(u)$ implies $(t_0\xi_\infty, \xi_\infty)\in \rmop{HWF}(e^{-it_0 H}u)$ for $u\in L^2(\mathbb{R}^n)$ and $t_0>0$. As a corollary, $e^{-itH}u=O(\jbracket{x}^{-\infty})$ ($\Gamma \ni x\to \infty$) for some $t>0$ and some conic neighborhood $\Gamma$ of the asymptotic momentum $\xi_\infty$ implies $(x_0, \xi_0)\not\in \rmop{WF}(u)$, which is proved by Craig, Kappeler and Strauss \cite{Craig-Kappeler-Strauss96}. 

K. Ito \cite{Ito06} generalizes this result to Euclidean spaces with asymptotically flat scattering metrics. Other studies on singularities of solutions to Schr\"odinger equations is in Doi \cite{Doi96} and Nakamura \cite{Nakamura09}. 

There are other concepts of wavefront sets for investigating propagation of singularities under Schr\"odinger equations. One of them is a Gabor wavefront set, defined in terms of Gabor transforms (also known as short-time Fourier transforms or wave packet transforms). Schulz and Wahlberg \cite{Schulz-Wahlberg17} proved the equality of homogeneous wavefront sets and Gabor wavefront sets. Gabor wavefront sets are studied in Cordero, Nicola and Rodino \cite{Cordero-Nicola-Rodino15}, Pravda-Starov, Rodino and Wahlberg \cite{Pravda-Starov-Rodino-Wahlberg18}. Other study by Gabor transforms is in Kato, Kobayashi and S. Ito \cite{Kato-Kobayashi-Ito15,Kato-Kobayashi-Ito17}. Another concept of wavefront sets is a quadratic scattering wavefront set, which is studied by Wunsch \cite{Wunsch99}. An equivalence of quadratic scattering wavefront sets and homogeneous wavefront sets is proved by K. Ito \cite{Ito06}. Melrose \cite{Melrose94} introduced scattering wavefront sets for investigating singularities at infinity. Analytic wavefront sets are also employed for an investigation of propagation of singularities under Schr\"odinger equations. They are studied by Robbiano and Zuily \cite{Robbiano-Zuily99, Robbiano-Zuily02}, Martinez, Nakamura and Sordoni \cite{Martinez-Nakamura-Sordoni09}. 

\subsection{Radially homogeneous wavefront sets on manifolds}

In the following, contrary to Section \ref{subs_hwf_euclid}, we employ pseudodifferential operators acting on \textit{half-densities}. We will briefly describe basic definition and properties on half-densities in Section \ref{subs_psido_half_densities}. 

We recall wavefront sets on manifolds: 

\begin{defi}[Wavefront sets on manifolds]
    Let $u\in L^2(M; \Omega^{1/2})$. $\rmop{WF}(u)$ is a subset of $T^*M\setminus 0$ defined as follows: $(x_0, \xi_0)\in T^*M\setminus 0$ is \textit{not} in $\rmop{WF}(u)$ if there exist local coordinate $\varphi: U\, (\subset M)\to V\, (\subset \mathbb{R}^n)$, $\chi\in C_c^\infty (U)$ and $a\in C_c^\infty (T^*M)$ such that $\chi=1$ near $x_0$, $a=1$ near $(x_0, \xi_0)$, $\rmop{supp}a \subset T^* U$ and 
    \[
        \| \chi\varphi^* (\tilde \varphi_*a)^\mathrm{w}(x, \hbar D)\varphi_*(\chi u)\|_{L^2}=O(\hbar^\infty). 
    \]
\end{defi}

Next we introduce radially homogeneous wavefront sets, which are introduced by K. Ito and Nakamura \cite{Ito-Nakamura09} to prove a microlocal smoothing effect on scattering manifolds. Before we introduce radially homogeneous wavefront sets on manifolds, we need to equip manifolds with some structure corresponding to the dilation $x\mapsto \hbar x$ on Euclidean spaces. In this paper, motivated by the fact that the dilation $x\mapsto \hbar x$ on Euclidean spaces is equivalent to $(r, \theta)\mapsto (\hbar r, \theta)$ in polar coordinates, we introduce a structure of ends of manifolds: 

\begin{assu}\label{assu_manifold_with_end_0}
    Let $n$ be the dimension of $M$. We assume that there exist an open subset $E$ of $M$, a compact manifold $S$ with dimension $n-1$, and a diffeomorphism $\Psi: E\to \mathbb{R}_+\times S$. Here $\mathbb{R}_+:=(0, \infty)$. We also assume that $M\setminus E$ is a compact subset of $M$.  
\end{assu}

The mapping $\Psi: E\to \mathbb{R}_+\times S$ in Assumption \ref{assu_manifold_with_end_0} induces the canonical mapping
\begin{equation}\label{eq_defi_psi_lift}
    \tilde\Psi: T^*E \longrightarrow T^*(\mathbb{R}_+ \times S), \quad 
        \tilde \Psi (x, \rho \diff r +\eta):=  (\Psi (x), \rho, \eta). 
\end{equation}

We introduce a class of functions dependent only on angular variables $\theta$ near infinity: 

\begin{defi}\label{defi_conical}
    A function $u\in C^\infty (M)$ is \textit{cylindrical} if there exist a constant $R\geq 1$ and a function $u_\mathrm{ang}\in C^\infty (S)$ such that $(u\circ \Psi^{-1})(r, \theta)=u_\mathrm{ang}(\theta)$ for all $r\geq R$. 
\end{defi}

\begin{exam*}
    \begin{itemize}
        \item All constant functions are cylindrical. 
        \item All $u\in C_c^\infty (M)$ are cylindrical by considering $u_\mathrm{ang}=0$. 
        \item The set of all cylindrical functions forms an algebra with respect to the natural sum, multiplication by complex numbers and product. 
    \end{itemize}
\end{exam*}

Throughout this paper, we use the term ``polar coordinates'' in the following sense: 

\begin{defi}
    We call $\varphi: U\, (\subset M) \to V\, (\subset \mathbb{R}^n)$ \textit{polar coordinates} if $\varphi$ is a local coordinate of the form $\varphi=(\rmop{id}\times \varphi^\prime)\circ \Psi$ where $\varphi^\prime: U^\prime\, (\subset S)\to V^\prime\, (\subset \mathbb{R}^{n-1})$ is a local coordinate on $S$. 
\end{defi}

Now we define radially homogeneous wavefront sets on manifolds. 

\begin{defi}[Radially homogeneous wavefront sets on manifolds]\label{defi_hwf_manifolds}
    For $u\in L^2(M; \Omega^{1/2})$, we define $\WFrh (u)$ as a subset of $T^*E$ defined as follows: $(x_0, \xi_0)\in T^*E$ is \textit{not} in $\WFrh (u)$ if there exist 
    \begin{itemize}
        \item a polar coordinate $\varphi: U \to V$, 
        \item a cylindrical function $\chi\in C^\infty (M)$ such that $\rmop{supp} \chi \subset U$ and $\Psi_*\chi(r, \theta)=1$ for large $r$ and $\theta$ near $\theta_0$ with $\Psi(x_0)=(r_0, \theta_0)$, and 
        \item  $a\in C_c^\infty (T^*V)$
    \end{itemize}
    such that $a=1$ near $\tilde\varphi (x_0, \xi_0)$ and 
    \begin{equation}\label{eq_hwf_manifold_definition}
        \| \chi\varphi^* a^\mathrm{w}(\hbar r, \theta, \hbar D_r, \hbar D_\theta)\varphi_*(\chi u)\|_{L^2}=O(\hbar^\infty). 
    \end{equation}
\end{defi}

\subsection{Main result}

Our subject is a Schr\"odinger equation for half-densities on manifolds: 
\begin{equation}
    \label{eq_schrodinger_manifold}
    i\frac{\partial}{\partial t}u(t, x)=Hu(t, x). 
\end{equation}
Here $H$ is a Hamiltonian of the form 
\[
    H=-\frac{1}{2}\triangle_g+V(x), 
\]
where $\triangle_g$ is the Laplace operator with respect to the Riemannian metric $g$ and $V$ is a real-valued smooth function with 
We further assume that 
\begin{equation}\label{eq_potential_bdd}
    |\partial_r^{\alpha_0}\partial_\theta^{\alpha^\prime} V(r, \theta)|\leq C_\alpha    
\end{equation}
holds for all multiindices $\alpha=(\alpha_0, \alpha^\prime)\in \mathbb{Z}_{\geq 0}\times \mathbb{Z}_{\geq 0}$ in polar coordinates $(r, \theta)$. 

$H$ acts on half-densities as 
\[
    H(\tilde u |\vol_g|^{1/2})=\left(-\frac{1}{2}\triangle_g \tilde u(x)+V(x)\tilde u(x)\right) |\vol_g|^{1/2}. 
\]
($|\vol_g|^{1/2}$ is the ``square root'' of the natural volume form $\vol_g=\sqrt{\det (g_{jk})}\diff x_1\wedge \cdots \wedge \diff x_n$ associated with the Riemannian metric $g$. We will explain details in Section \ref{subs_psido_half_densities}.)

We will explain our assumptions concretely in the following, but we emphasize that our setting includes not only the cases of asymptotically conical manifolds, but also those of asymptotically hyperbolic manifolds. 

\begin{rema*}
    We do not assume $\partial_x^\alpha V =O(\jbracket{x}^{2-\varepsilon})$ for some $\varepsilon>0$, but the boundedness \eqref{eq_potential_bdd} in order to argue in the symbol class $S^m_\mathrm{cyl}(T^*M)$ (introduced in Section \ref{subs_cylindrical_class}), which do not allow any growth in spatial direction ($r\to \infty$). It may be possible to introduce suitable classes of symbols with spatial growth and treat such potentials $V$, we restrict ourselves to the case of bounded potentials for simplicity. 
\end{rema*}

We take suitable polar coordinates such that the vector $\partial_r$ and the tangent space $T_{(r, \theta)}S$ intersect orthogonally: 

\begin{assu}\label{assu_manifold_with_end_II}
    Under Assumption \ref{assu_manifold_with_end_0}, the Riemannian metric $g$ has the representation 
\begin{equation}\label{eq_metric_polar}
   g=\Psi^*(c(r, \theta)^2 \diff r^2+h(r, \theta, \diff \theta)), \quad \text{for } (r, \theta)\in \mathbb{R}_+ \times S \simeq E, 
\end{equation}
where $c: E\to \mathbb{R}_+$ is a smooth function and $h(r, \theta, \diff \theta)$ is a metric on $S$ dependent smoothly on the radial variable $r$. 
\end{assu}

We will see in Section \ref{subs_escape_function} a construction of such diffeomorphism $\Psi: E\to \mathbb{R}_+ \times S$ by escape functions, which is a generalization of the function $|x|$ in $\mathbb{R}^n$. 

We further assume a compatibility condition of the diffeomorphism $\Psi: E\to \mathbb{R}_+\times S$ and the Riemannian metric . Let $h^*(r, \theta, \eta)$ be the fiber metric on $T^*S$ induced by $h(r, \theta, \diff \theta)$. More explicitly, if $h(r, \theta, \diff \theta)$ has a local representation
\[
    h(r, \theta, \diff \theta)=\sum_{i, j=1}^{n-1}h_{ij}(r, \theta)\diff \theta_i \diff \theta_j, 
\]
then we define 
\[
    h^*(r, \theta, \eta):=\sum_{i, j=1}^{n-1}h^{ij}(r, \theta)\eta_i \eta_j, 
\]
where $(h^{ij}(r, \theta))$ is the inverse matrix of $(h_{ij}(r, \theta))$ and $\eta=\eta_1 \diff \theta_1+\cdots +\eta_{n-1}\diff \theta_{n-1}$. Furthermore we define 
\[
    |\xi|_{g^*}^2:=c(r, \theta)^{-2}\rho^2+h^*(r, \theta, \eta)
\]
for $\xi=\rho \,\diff r+\eta\in T^*_r \mathbb{R}_+\oplus T^*_\theta S$. $|\xi|_{g^*}$ is the norm of $\xi\in T^*M$ with respect to the fiber metric $g^*$ on $T^*M$ induced by the metric $g$. We define a free Hamiltonian $h_0(x, \xi)$ as 
\begin{equation}\label{eq_free_hamiltonian}
    h_0(x, \xi):=\frac{1}{2}|\xi|_{g^*}^2. 
\end{equation}

\begin{assu}\label{assu_classical}
    There exist a function $f(r)$ and constants $c_0>1/2$, $C>0$ and $\mu>0$ such that the following properties hold. 
    \begin{enumerate}[label=(\roman*)]
        \item $f$ is positive, belongs to $C^1$ class and the inequality \begin{equation}\label{eq_ineq_f_logbdd}
            c_0r^{-1}\leq \frac{f^\prime (r)}{f(r)}\leq C
        \end{equation} 
        holds for all $r\geq 1$. 
        \item The inequality 
        \begin{equation}\label{eq_ineq_model_bdd}
            C^{-1}f(r)^{-2}h^*(1, \theta, \eta)\leq h^*(r, \theta, \eta)\leq Cf(r)^{-2}h^*(1, \theta, \eta)
        \end{equation} 
        holds for all $(r, \theta, \eta)\in [1, \infty)\times T^*S$.  
        \item \label{assu_sub_dtheta_h}The inequality 
        \begin{equation}\label{eq_ineq_angular_bdd}
            |\partial_\theta h^{ij}(r, \theta, \eta)|\leq Ch^*(r, \theta, \eta)
        \end{equation} 
        holds for all $(r ,\theta, \eta)\in [1, \infty)\times T^*S$. 
        \item (Classical analogue of Mourre estimate) The estimate
        \begin{equation}\label{eq_classical_mourre} \{ f\rho, h_0\} \geq 2f^\prime (r)(h_0(r, \theta, \eta)-Cr^{-1-\mu}),\end{equation}
        holds for all $(r, \theta, \rho, \eta)\in T^*(\mathbb{R}_+\times S)$. Here $\{ \cdot, \cdot\}$ is the Poisson bracket on $T^*M$: 
        \[
            \{ a, b\}:=\sum_{i=1}^n \left( \frac{\partial a}{\partial x_i}\frac{\partial b}{\partial \xi_i}-\frac{\partial a}{\partial \xi_i}\frac{\partial b}{\partial x_i}\right). 
        \]
        \item (Short range conditions)\label{assu_sub_short_range} $|c-1|$ and $|\partial_\theta c|$ are at most $O(r^{-1-\mu})$ as $r\to\infty$. 
    \end{enumerate}
\end{assu}

\begin{rema*}
    The size of $|\partial_\theta h^*(r, \theta, \eta)|$ in \eqref{eq_ineq_angular_bdd} and $|\partial_\theta V|$ and $|\partial_\theta c|$ in \ref{assu_sub_short_range} of Assumption \ref{assu_classical} are measured by the metric $h(1, \theta, \diff \theta)$. 
\end{rema*}

\begin{exam*}[model manifolds]
    Assumption \ref{assu_classical} is a generalization of model cases $g=\diff r^2+f(r)^2 h(\theta, \diff \theta)$. For instance $f(r)$ is the form $f(r)=r^a, e^{b r^c}$ ($a>1/2$, $b>0$, $0<c\leq 1$) for $r\geq 1$. The free Hamiltonian \eqref{eq_free_hamiltonian} with respect to this metric becomes
    \[ h_0=\frac{1}{2}(\rho^2+f(r)^{-2}h^*(\theta, \eta)). \]
    The Poisson bracket $\{f \rho, h_0\}=-\partial_r h_0$ is
    \[ \{ f\rho, h_0\}=f^\prime(r)h_0(\theta, \eta). \]
    This is a prototype of classical Mourre type estimate in Assumption \ref{assu_classical}. 
\end{exam*}

Then, for any nontrapping classical orbit $\Psi^{-1}(r(t), \theta(t), \rho(t), \eta(t))$ ($r(t)\to \infty$) with respect to the free Hamiltonian $h_0$, the limit 
\[
    (\rho_\infty, \theta_\infty, \eta_\infty):=\lim_{t\to\infty} (\rho(t), \theta(t), \eta(t))\in \mathbb{R}_+\times T^*S
\] 
exists under Assumption \ref{assu_classical}. We state this more precisely in Theorem \ref{theo_classical_estimate}. We remark that the classical analogue of Mourre estimate \eqref{eq_classical_mourre} plays an essential role in a proof of Theorem \ref{theo_classical_estimate}. The inequality \eqref{eq_classical_mourre} insures that the classical orbit $(x(t), \xi(t))$ approaches an asymptotic orbit $\Psi^{-1}(\rho_\infty t, \theta_\infty, \rho_\infty, \eta_\infty)$ rapidly. 

\begin{rema*}
    It is well known that the inequality 
    \begin{equation}\label{eq_mourre}
        1_I(H)i[H, A]1_I(H)\geq c1_I(H)-1_I(H)K 1_I(H), 
    \end{equation}
    where 
    \begin{itemize}
        \item $A$, $H$ are self-adjoint operators on an abstract Hilbert space $\mathcal{H}$, 
        \item $I=(a, b)$ is a bounded open interval and $1_I:\mathbb{R}\to \{0, 1\}$, is the indicator function of $I$, 
        \item $c$ is a positive constant, and 
        \item $K$ is a compact operator on $\mathcal{H}$, 
    \end{itemize}
    plays an important role in a quantum scattering theory. The inequality \eqref{eq_mourre} is called the \textit{Mourre estimate} \cite{Mourre80}. A typical case is $\mathcal{H}=L^2(\mathbb{R}^n)$, $H=-\triangle/2$ and $A=(x\cdot D_x+D_x\cdot x)/2$. The principal symbol of $A$ is equal to $x\cdot \xi=r\rho$ in polar coordinates, and the principal symbol of $i[H, A]$ is $\{ r\rho, h_0\}$. Thus we can regard \eqref{eq_classical_mourre} as a classical analogue of \eqref{eq_mourre} with $f(r)=r$. 
\end{rema*}

The last assumption is a boundedness of quantities related to the metric necessary for applying microlocal analysis. 

\begin{assu}\label{assu_higher_derivative}
    For all multiindices $\alpha\in \mathbb{Z}_{\geq 0}$ with $|\alpha|\geq 1$, there exists $C>0$ that the inequalities
    \[
        |\partial_{r, \theta}^\alpha h^*(r, \theta, \eta)|\leq Ch^*(r, \theta, \eta)
    \]
    and  
    \[
        |\partial_{r, \theta}^\alpha c(r, \theta)|\leq Cr^{-1-\mu}
    \]
    hold for all $(r ,\theta, \eta)\in [1, \infty)\times T^*S$. 
\end{assu}

Now we state our main theorem. 

\begin{theo}\label{theo_main_rhwf}
    Suppose Assumption \ref{assu_manifold_with_end_0}--\ref{assu_higher_derivative}. Let $u\in L^2(M; \Omega^{1/2})$ and $t_0>0$. Let $(x(t), \xi (t))$ be a nontrapping classical orbit with initial point $(x_0, \xi_0)\in T^*M$ with respect to the free Hamiltonian $h_0$ and $(\rho_\infty, \theta_\infty, \eta_\infty)\in \mathbb{R}\times T^*S$ be the asymptotic angle and momentum. Then $(x_0, \xi_0)\in \rmop{WF}(u)$ implies $\tilde\Psi^{-1}(\rho_\infty t_0, \theta_\infty, \rho_\infty, \eta_\infty)\in \WFrh (e^{-it_0H}u)$. 
\end{theo}

We emphasize that the proof of the main theorem in the case of asymptotically conical/hyperbolic is unified. 

In the case of $M=\mathbb{R}^n$, $S=S^{n-1}$ (($n-1$)-dimensional unit sphere) and $\Psi(x):=(|x|, x/|x|)$, we have a relation between radially homogeneous wavefront sets and homogeneous wavefront sets. First we remark a characterization of homogeneous wavefront sets by polar coordinates: 
\begin{prop}\label{prop_hwf_polar}
    For $u\in L^2 (M; \Omega^{1/2})$, $x\neq 0$ and $\xi\in\mathbb{R}^n$, the following statements are equivalent. 
    \begin{enumerate}[label=(\roman*)]
        \item $(x, \xi)\not\in \rmop{HWF}(u)$. 
        \item There exist polar coordinates $\varphi: U\to V$, cylindrical function $\chi\in C^\infty (\mathbb{R}^n)$ with $\rmop{supp}\chi \subset U$ and $\chi=1$ near the set $\{ \lambda x \mid \lambda\geq 1\}$, and $a\in C_c^\infty (V)$ with $a=1$ near $\Psi (x, \xi)$ such that 
        \begin{equation}\label{eq_hwf_polar}
        \|a^\mathrm{w}(\hbar r, \theta, \hbar D_r, \hbar^2 D_\theta)\varphi_*(\chi u)\|_{L^2(\mathbb{R}^n; \Omega^{1/2})}=O(\hbar^\infty)
        \end{equation}
        holds. 
    \end{enumerate}
\end{prop}

We compare symbols $a(\hbar r, \theta, \hbar \rho, \hbar \eta)$ in \eqref{eq_hwf_manifold_definition} and $a(\hbar r, \theta, \hbar \rho, \hbar^2 \eta)$ in \eqref{eq_hwf_polar}. Let us pay attention to the $\eta$ variable. The support of $a(\hbar, \theta, \hbar \rho, \hbar \eta)$ is included in $|\eta-\hbar^{-1}\eta_0|\leq O(\hbar^{-1})$, whereas that of $a(\hbar, \theta, \hbar \rho, \hbar^2 \eta)$ is included in $|\eta-\hbar^{-2}\eta_0|\leq O(\hbar^{-2})$ ($\tilde\Psi(x_0, \xi_0)=(r_0, \theta_0, \rho_0, \eta_0)$). In particular, if $\eta_0=0$, then the support of $a(\hbar r, \theta, \hbar \rho, \hbar \eta)$ is included in the level set $\{a(\hbar r, \theta, \hbar \rho, \hbar^2 \eta)\}$ for sufficiently small $\hbar>0$. Thus we have the following corollary, noting that $\tilde\Psi^{-1}(r_0, \theta_0, \rho_0, 0)=(x_0, (\xi_0\cdot \hat x_0)\hat x_0)$ with $\hat x_0:=x_0/|x_0|$: 
\begin{coro}\label{coro_cylindrical_homogeneous}
    Let $u\in L^2(\mathbb{R}^n; \Omega^{1/2})$. We define a homogeneous wavefront set of half-density $u=\tilde u |\diff x|^{1/2}$ as that of the function $\tilde u(x)$. Then, for $x\neq 0$, $(x, \xi)\in \WFrh (u)$ implies $(x, (\xi\cdot \hat x) \hat x)\in \rmop{HWF}(u)$ where $\hat x:=x/|x|$. 
\end{coro}
We also prove Corollary \ref{coro_cylindrical_homogeneous} in Section \ref{subs_cylindrical_homogeneous}. Combining Theorem \ref{theo_main_rhwf} and Corollary \ref{coro_cylindrical_homogeneous}, we immediately obtain a propagation of homogeneous wavefront sets on Euclidean spaces: 

\begin{coro}
    Let $M=\mathbb{R}^n$ with the usual Euclidean metric, $\Psi(x)=(x, x/|x|)$, and $H=-\triangle/2+V$ with $|\partial_r^{\alpha_0}\partial_\theta^{\alpha^\prime}V|\leq C_\alpha$ for all multiindices $\alpha=(\alpha_0, \alpha^\prime)\in \mathbb{Z}_{\geq 0}\times \mathbb{Z}_{\geq 0}^{n-1}$. Let $(x(t), \xi(t))=(x(0)+t\xi_\infty, \xi_\infty)$ be a free classical orbit. Then, for $u\in L^2(\mathbb{R}^n; \Omega^{1/2})$ and $t_0>0$, $(x(0), \xi(0))\in\rmop{WF}(u)$ implies $(t_0\xi_\infty, \xi_\infty)\in \rmop{HWF}(e^{-it_0 H}u)$. 
\end{coro}

There are many studies on Schr\"odinger equations on manifolds. For example, Schr\"odinger propagator on scattering manifolds are studied by Hassell and Wunsch \cite{Hassell-Wunsch05}. K. Ito and Nakamura \cite{Ito-Nakamura09} generalized the result of \cite{Hassell-Wunsch05}. Microlocal analysis on asymptotically hyperbolic spaces are studied by, for instance, Bouclet \cite{Bouclet11-1,Bouclet11-2}, S\'{a} Barreto \cite{SaBarreto05}, Melrose, S\'{a} Barreto and Vasy \cite{Melrose-SaBaretto-Vasy14}. Our idea of microlocal analysis in polar coordinates is inspired by Bouclet \cite{Bouclet11-1,Bouclet11-2}.

We describe outline of proof of the main Theorem \ref{theo_main_rhwf} in Section \ref{sect_outline_proof}. We reduce the proof of main theorem to three key propositions (Theorem \ref{theo_symbol_aim}, Theorem \ref{theo_hwf_quantization} and Proposition \ref{prop_wf_quantization}) there. In Section \ref{sect_classical_mechanics}, we prove the existence of asymptotic angle and momentum $(\theta_\infty, \rho_\infty, \eta_\infty)$. We develop a pseudodifferential calculus necessary for our aim in Section \ref{sect_psido_manifolds}. In particular, we prove two of key propositions (Theorem \ref{theo_hwf_quantization} and Proposition \ref{prop_wf_quantization}) in Section \ref{subs_quantization_wf_hwf}. Finally, in Section \ref{sect_heisenberg_derivative}, we estimate Heisenberg derivatives of operators constructed in Section \ref{sect_outline_proof} and prove the rest key proposition (Theorem \ref{theo_symbol_aim}). 

\fstep{Notation}For derivatives, we use notations $D_{x_j}:=-i\partial_{x_j}$ and multiindices $D^\alpha:=D_{x_1}^{\alpha_1}\cdots D_{x_n}^{\alpha_n}$. We also denote $\jbracket{x}:=\sqrt{1+|x|^2}$. As in the definition \eqref{eq_defi_psi_lift} of $\tilde\Psi$ , for a diffeomorphism $\varphi: U\to V$, we denote the canonical mapping associated with $\varphi$ by $\tilde\varphi: T^*U \to T^*V$, $\tilde\varphi (x, \xi):=(\varphi (x), d\varphi^{-1}(\varphi (x))^* \xi)$. 

\section{Outline of proof}\label{sect_outline_proof}

We prove our main theorem by following the argument in Nakamura \cite{Nakamura05}. 

We construct symbols connecting wave functions at time $t=0$ and those at time $t=t_0$ by the following procedure. 

\initstep\step\label{step_construction} We take a function $\chi\in C_c^\infty ([0, \infty))$ such that $\chi\geq 0$, $\chi^\prime (x)\leq 0$ and 
\[
    \chi(x)=
    \begin{cases}
        1 & \text{if } x\leq 1, \\
        0 & \text{if } x\geq 2. 
    \end{cases}\]
Since $\chi$ is constant near $x=0$, $\chi$ belongs to $C^\infty(\mathbb{R})$. Take polar coordinates $\varphi: U\to V$ near $\Psi^{-1}(r_0, \theta_\infty)$ where $r_0\gg 1$. We take sufficiently small constants $\delta_0, \delta_1, \delta_2, \ldots>0$ and $\lambda \in (0, 1]$ and consider 
\begin{equation}\label{eq_defi_chi4}
    \underbrace{\chi\left(\frac{|r-r(t)|}{4\delta_j t}\right)}_{=:\chi_{1j}} 
    \underbrace{\chi\left(\frac{|\theta-\theta(t)|}{\delta_j-t^{-\lambda}}\right)}_{=:\chi_{2j}}
    \underbrace{\chi\left(\frac{|\rho-\rho(t)|}{\delta_j-t^{-\lambda}}\right)}_{=:\chi_{3j}}
    \underbrace{\chi\left(\frac{|\eta-\eta(t)|}{\delta_j -t^{-\lambda}} \right)}_{=:\chi_{4j}} 
\end{equation}
for sufficiently large $t$. Denote the range of $t$ by $[T_0, \infty)$ with a sufficiently large $T_0>0$. We pull $\chi_{1j}\chi_{2j}\chi_{3j} \chi_{4j}$ back by the canonical coordinates $\tilde\varphi: T^*U\to T^*V$ induced by $\varphi: U\to V$ and define 
\begin{equation} \label{eq_defi_psi-1}
    \tilde\psi_j(t, x, \xi):=\tilde\varphi^*(\chi_{1j} \chi_{2j} \chi_{3j} \chi_{4j})\in C^\infty (T^*U; \mathbb{R}). 
\end{equation}
Since the support of $\tilde\psi_j$ are included in $T^*U$, we extend $\tilde\psi_j$ to a smooth function on $T^*M$ by defining $\tilde\psi_j=0$ outside $T^*U$. 

\step We take a cutoff function $\alpha\in C^\infty (\mathbb{R})$ which satisfies $\alpha^\prime\geq 0$ and
\[\alpha(t)=
\begin{cases}
    0 & \text{if } t\leq T_0, \\
    1 & \text{if } t\geq T_0+1. 
\end{cases}\]
We define $\psi_0(t, x, \xi)$ as a solution to a transport equation
\begin{equation}\label{eq_transport}
    \begin{split}
        &\partial_t \psi_j +\{ \psi_j, h_0\}=\alpha(t)(\partial_t \tilde\psi_j +\{ \tilde\psi_j, h_0\}), \\
        &\psi_j(T_0+1, x, \xi)=\tilde\psi_j(T_0+1, x, \xi). 
    \end{split}
    \end{equation}
    
\step We choose positive constants $c_1, c_2, \ldots$ and construct a symbol $\tilde a(\hbar, t, r, \theta, \rho, \eta)$ such that 
\begin{equation}\label{eq_definition_symbol}
    \tilde a(\hbar, t)\sim \sum_{j=1}^\infty c_j\hbar^j t \psi_j(\hbar, t)
\end{equation}
    by the Borel theorem. We quantize symbols $\psi_0(\hbar^{-1}t)$ and $\tilde a(\hbar; \hbar^{-1}t)$ (the procedure of quantization is in Definition \ref{defi_quantization}) and obtain quantized operators $\Oph \psi_0(\hbar^{-1}t)$ and $\Oph (\tilde a(\hbar, \hbar^{-1}t))$. We define an operator $A_\hbar (t)$ as 
\begin{equation}\label{eq_defi_op_at}
     A_\hbar (t):=\Oph (\psi_0(\hbar^{-1}t))^*\Oph (\psi_0(\hbar^{-1}t))+\Oph (\tilde a(\hbar, \hbar^{-1}t)). 
\end{equation}

Now we state two key lemmas for the proof of the main theorem (Theorem \ref{theo_main_rhwf}). The first lemma states the positivity of the Heisenberg derivative modulo $O(\hbar^\infty)$ of the time-dependent operator $A_\hbar (t)$. 

\begin{theo}\label{theo_symbol_aim}
If we take suitable $\delta_0, \delta_1, \delta_2, \ldots>0$, $\lambda \in (0, 1]$, $T_0>0$ and $c_1, c_2, \ldots>0$ in above construction procedure, then the following statements hold. 
\begin{enumerate}[label=(\roman*)]
    \item $a(\hbar, t)\in S^{-2}_\mathrm{cyl}(T^*M)$ and forms a bounded family in $S^{-2}_\mathrm{cyl}(T^*M)$. 
    \item For any $k\geq 0$, the inequality
    \begin{equation}\label{eq_positive_heisenberg_derivative}\partial_t A_\hbar (t)-i[A_\hbar (t), H]\geq O_{L^2\to L^2}(\hbar^k)
    \end{equation}
    holds uniformly in $0\leq t \leq t_0$. 
\end{enumerate}
\end{theo}

The second lemma states that the operator which appeared in the definition of radially homogeneous wavefront sets (Definition \ref{defi_hwf_manifolds}) is approximated by $A_\hbar (t)$. 

\begin{theo}\label{theo_hwf_quantization}
    Let polar coordinates $\varphi: U\to V$ in $E$, $a\in C_c^\infty (T^*V)$ and a cylindrical function $\chi\in C^\infty (M)$ satisfy conditions in Definition \ref{defi_hwf_manifolds} except for \eqref{eq_hwf_manifold_definition}. Then, if we take sufficiently small $\delta_0, \delta_1, \cdots >0$ in Step \ref{step_construction} properly, then we have
    \[
        A_\hbar (t_0)-A_\hbar (t_0)\chi \varphi^* a^\mathrm{w}(\hbar r, \theta, \hbar D_r, \hbar D_\theta)(\varphi_*\chi)\varphi_*=O_{L^2\to L^2}(\hbar^\infty). 
    \]
\end{theo}

In addition to the above key lemmas, we state a technical lemma on the wavefront sets in order to describe the proof of the main theorem briefly. 

\begin{prop}\label{prop_wf_quantization}
    Let $(x_0, \xi_0)\in T^*M\setminus 0$ and $u\in L^2(M; \Omega^{1/2})$. If there exists a symbol $a\in C_c^\infty (T^*M)$ such that $a=1$ near $(x_0, \xi_0)$ and $\Oph (a)u=O_{L^2}(\hbar^\infty)$, then $(x_0, \xi_0)\not\in \rmop{WF}(u)$. 
\end{prop}

We can prove Theorem \ref{theo_main_rhwf} from Theorem \ref{theo_symbol_aim}, Proposition \ref{prop_wf_quantization} and Theorem \ref{theo_hwf_quantization} as follows. 

\begin{proof}[Proof of Theorem \ref{theo_main_rhwf}]
Assume $(\rho_\infty t_0, \theta_\infty, \rho_\infty, \eta_\infty)\not\in \WFrh (e^{-it_0H}u)$. By considering $\partial_t (e^{itH}A_\hbar (t)e^{-itH})=e^{itH}(\partial_t A_\hbar (t)-i[A_\hbar (t), H])e^{-itH}$, we have 
\begin{equation}\label{eq_expectation_estimate} 
    \begin{split}
    \jbracket{A_\hbar (0)u, u}
    &=\jbracket{A_\hbar (t_0)e^{-it_0H}u, e^{-it_0H}u} \\
    &\quad -\int_0^{t_0} \jbracket{(\partial_t A_\hbar (s)-i[A_\hbar (s), H])e^{-isH}u, e^{-isH}u}\, \diff s \\
    &\leq \jbracket{A_\hbar (t_0)e^{-it_0H}u, e^{-it_0H}u}+O(\hbar^k)
    \end{split}
\end{equation}
for any $k\geq 0$ by \eqref{eq_positive_heisenberg_derivative}. We consider 
$\jbracket{A_\hbar (0)u, u}$ and $\jbracket{A_\hbar (t_0)e^{-it_0H}u, e^{-it_0H}u}$ respectively. 

\fstep{$\bm{\jbracket{A_\hbar (0)u, u}}$}Since $A_\hbar(0)=\Oph (\psi_0(0))^*\Oph (\psi_0(0))$ by \eqref{eq_definition_symbol}, we have 
\begin{equation}\label{eq_initial_wf}
     \jbracket{A_\hbar (0)u, u}=\|\Oph (\psi_0(0))u\|_{L^2}^2. 
\end{equation}
    
\fstep{$\bm{\jbracket{A_\hbar (t_0)e^{-it_0H}u, e^{-it_0H}u}}$}We set $(x_\infty, \xi_\infty):=\tilde\Psi^{-1} (\rho_\infty t_0, \theta_\infty, \rho_\infty, \eta_\infty)$. Since $(x_\infty, \xi_\infty)\not\in \WFrh (e^{-it_0 H}u)$, there exist a polar coordinate $\varphi: U\to V$, cylindrical cutoff $\chi\in C^\infty (M)$ and $a\in C_c^\infty (T^*V)$ which satisfy the conditions in Definition \ref{defi_hwf_manifolds}. 
We put $B_\hbar =\chi\varphi^*b^\mathrm{w}(\hbar r, \theta, \hbar D_r, \hbar D_\theta)(\varphi_*\chi)\varphi_*$. By Theorem \ref{theo_hwf_quantization}, we take sufficiently small $\delta_0, \delta_1, \cdots>0$ properly such that $A_\hbar (t_0)-A_\hbar (t_0)B_\hbar=O_{L^2\to L^2}(\hbar^\infty)$. Since $B_\hbar e^{-it_0 H}u=O_{L^2}(\infty^\infty)$ by the definition of radially homogeneous wavefront sets (Definition \ref{defi_hwf_manifolds}), we have 
\begin{equation}
    \label{eq_at0_neg}
    A_\hbar (t_0)e^{-it_0 H}u=O_{L^2}(\hbar^\infty). 
\end{equation}

\fstep{Conclusion}Combining \eqref{eq_initial_wf}, \eqref{eq_expectation_estimate} and \eqref{eq_at0_neg}, we have 
\[
    \| \Oph (\psi_0(0))u\|_{L^2}^2\leq O(\hbar^\infty). 
\]
We recall that $(x(t), \xi(t))$ is a solution to Hamilton equation with respect to the free Hamiltonian \eqref{eq_free_hamiltonian}. Then 
\[
    \frac{\diff}{\diff t}(\psi_{-1}(t, x(t), \xi(t)))=(\partial_t \psi_{-1} +\{ \psi_{-1}, h_0\})(t, x(t), \xi(t)). 
\]
In particular $\psi_{-1}(t, x(t), \xi(t))=1$ for $t\geq T_0$ in this case. Thus 
\[
    (\partial_t \psi_{-1} +\{ \psi_{-1}, h_0\})(t, x(t), \xi(t))=0, \quad \forall t\geq T_0. 
\]
Hence $\psi_0(t, x(t), \xi(t))$ is a solution to the initial value problem $\partial_t (\psi_0(t, x(t), \xi(t)))=0$, $\psi_0(T_0+1, x(T_0+1), \xi(T_0+1))=1$. Thus $\psi_0(t, x(t), \xi(t))=1$ for all $t\geq T_0$. In particular we have $\psi_0(0, x(0), \xi(0))=1$. Since flows are families of diffeomorphisms generally, we have $\psi_0(0, x, \xi)=1$ near $(x(0), \xi(0))$. Hence by Proposition \ref{prop_wf_quantization}, we obtain $(x_0, \xi_0)\not\in\rmop{WF} (e^{-it_0 H}u)$. 
\end{proof}

\section{Classical mechanics}\label{sect_classical_mechanics}

The only purpose of this section is the proof of the existence of asymptotic angle and momentum: 

\begin{theo}\label{theo_classical_estimate}
    Assume Assumption \ref{assu_classical}. Let $\Psi^{-1}(r(t), \theta(t), \rho(t), \eta(t))$ be a nontrapping classical orbit (i.e., $r(t)\to\infty$ as $t\to \infty$) with respect to the free Hamiltonian \eqref{eq_free_hamiltonian}. Then the asymptotic angle and momentum $(\rho_\infty, \theta_\infty, \eta_\infty):=\lim_{t\to\infty} (\rho(t), \theta(t), \eta(t))\in (0, \infty)\times T^*S$ exists. Moreover, there exists a constant $C>0$ such that the estimate 
    \[ \rho_\infty t-C \leq r(t) \leq \rho_\infty t+C\]
    holds for all sufficiently large $t$. 
\end{theo}

The proof of Theorem \ref{theo_classical_estimate} is separated to several steps. In the following, $(r(t), \theta(t), \rho(t), \eta(t))$ is a classical orbit satisfying the assumption of Theorem \ref{theo_classical_estimate}. 

We record the explicit form of the radial component of the Hamilton equation with respect to the free Hamiltonian \eqref{eq_free_hamiltonian}: 
\begin{equation}
    \label{eq_hamilton_equation_radial}
    \frac{\diff r}{\diff t}=c(r, \theta)^{-2}\rho, \quad \frac{d\rho}{\diff t}=\frac{\partial_r c}{c^3}\rho^2-\frac{1}{2}\frac{\partial}{\partial r}h^*(r, \theta, \eta)
\end{equation}
We also note the energy conservation law $h_0 (x(t), \xi (t))=E_0$. The total energy $E_0$ is positive by the nontrapping condition.

\begin{lemm}
    \label{lemm_asymp_radial}
    The asymptotic radial momentum $\rho_\infty:=\lim_{t\to\infty}\rho(t)$ exists and equals to $\sqrt{2E_0}$. 
\end{lemm}

\begin{proof}
    \initstep\step We prove that $\rho(t)>0$ for sufficiently large $t$. The classical Mourre type estimate \eqref{eq_classical_mourre} implies 
    \begin{equation}\label{eq_frho_0}
        \frac{\diff}{\diff t}(f(r(t))\rho(t)) \geq f^\prime (r(t))(2E_0-Cr(t)^{-1-\mu})
    \end{equation}
By $f^\prime>0$, the nontrapping condition $r(t)\to \infty$ and $E_0>0$, the right hand side of \eqref{eq_frho_0} is positive for sufficiently large $t$. Thus $f(r(t))\rho (t) \geq f(r(T))\rho (T)$ for $t\geq T\gg 1$. 

We have to find a large $T$ such that $f(r(T))\rho (T)>0$. Suppose that there exists $T^\prime>0$ such that $\rho (T)\leq 0$ for all $T\geq T^\prime$. Then the Hamilton equation \eqref{eq_hamilton_equation_radial} implies 
\[
    r(T)=r(T^\prime)+\int_{T^\prime}^T c(r(t), \theta (t))^{-2}\rho (t)\, \diff t\leq r(T^\prime)\leq 0
\]
for all $T\geq T^\prime$. This contradicts to the nontrapping condition $r(T)\to \infty$. Thus, for any $T^\prime>0$, there exists $T\geq T^\prime$ such that $\rho(T)>0$. Hence we obtain $\rho (t)\geq f(r(T))\rho (T)/f(r(t))>0$ for all $t\geq T$. 

\step We employ the classical Mourre type estimate \eqref{eq_classical_mourre} again. For $\rho >0$ and $r\gg 1$, we have
\[
    \{f\rho, h_0\}\geq f^\prime (r)(2h_0-Cr^{-1-\mu})\geq \frac{2h_0-Cr^{-1-\mu}}{c(r, \theta)^{-1}\sqrt{2h_0}}\{ f, h_0\}
\]
by $c^{-1}\rho \leq \sqrt{2h_0}$. 
Thus 
\begin{equation}
    \label{eq_dt_rho_below}
        \frac{\diff}{\diff t}(f(r(t))\rho (t))\geq \frac{2E_0-Cr(t)^{-1-\mu}}{c(r(t), \theta(t))^{-1}\sqrt{2E_0}}\frac{\diff}{\diff t}(f(r(t))). 
\end{equation}
Take an arbitrary small $\varepsilon>0$. Since 
\[
    \lim_{t\to\infty} \frac{2E_0-Cr(t)^{-1-\mu}}{c(r(t), \theta(t))^{-1}\sqrt{2E_0}}
    =\sqrt{2E_0}, 
\]
\eqref{eq_dt_rho_below} implies 
\[
    \frac{\diff}{\diff t}(f(r(t))\rho (t))\geq ( \sqrt{2E_0}-\varepsilon)\frac{\diff}{\diff t}(f(r(t)))
\]
for sufficiently large $t$. This differential inequality shows 
\[
    f(r(t))\rho(t)\geq ( \sqrt{2E_0}-\varepsilon)f(r(t))+C
\]
for some constant $C>0$ and sufficiently large $t$. Dividing both sides by $f(r(t))>0$ and taking a infimum limit as $t\to \infty$, we have 
\[
    \liminf_{t\to\infty} \rho (t)\geq \sqrt{2E_0}-\varepsilon
\]
by $f(r)\geq C^{-1}r^{c_0}$ and $r(t)\to \infty$ as $t\to\infty$. Since $\varepsilon>0$ is arbitrary, we can take a limit $\varepsilon\to 0$ and obtain $\liminf_{t\to\infty} \rho (t)\geq 2E_0$. Combining this with $\rho \leq c(r, \theta)\sqrt{2h_0}$ and $\lim_{r\to\infty}c(r, \theta)=1$ (Assumption \ref{assu_classical} \ref{assu_sub_short_range}), we obtain 
\[
    \sqrt{2E_0}\leq \liminf_{t\to\infty} \rho(t)\leq \limsup_{t\to \infty} \rho (t)\leq \limsup_{t\to\infty} c(r(t), \theta (t))\sqrt{2E_0}=\sqrt{2E_0}. \qedhere
\]
\end{proof}

\begin{lemm}\label{lemm_rho1}
    We have an asymptotic behavior 
    \begin{equation}\label{eq_lemm_rho1}
        \rho(t)=\rho_\infty+O(t^{-1-\delta}) \quad (t\to\infty)
    \end{equation}
    for any $0<\delta<\min\{ 2c_0-1, \mu\}$. 
\end{lemm}

\begin{proof}
    We define $a(r, \rho):=f(r)^2(\sqrt{2h_0}-\rho)$. A direct calculation shows that 
    \begin{equation}\label{eq_rho_1}
        \{ a, h_0 \}
        =2\frac{ c^{-2}f^\prime}{f}\rho a-f^2\{\rho, h_0\}.  
    \end{equation}

    A simple calculation shows that classical Mourre estimate \eqref{eq_classical_mourre} is equivalent to 
\begin{equation}
    \label{eq_classical_mourre_alt}
    \{ \rho, h_0\} \geq \frac{f^\prime (r)}{f(r)}\left(2h_0-c^{-2}\rho^2-Cr^{-1-\mu}\right). 
\end{equation}

By \eqref{eq_classical_mourre_alt}, we have     
    \begin{equation}\label{eq_rho_2}
        \begin{split}
        \{ \rho , h_0\} 
        &\geq \frac{f^\prime}{f}(2h_0-c^{-2}\rho^2-Cr^{-1-\mu}) \\
        &=\frac{c^{-2}f^\prime}{f}(\sqrt{2h_0}+\rho)(\sqrt{2h_0}-\rho)-C\frac{f^\prime r^{-1-\mu}}{f}+\frac{2h_0 f^\prime}{f}(1-c^{-2}) \\
        &\geq \frac{c^{-2}f^\prime}{f^3}(\sqrt{2h_0}+\rho)a-C\frac{f^\prime r^{-1-\mu}}{f}(1+h_0). 
        \end{split}
    \end{equation}
    Here we used the short range condition $|c-1|=O(r^{-1-\mu})$ (Assumption \ref{assu_classical} \ref{assu_sub_short_range}). Combining \eqref{eq_rho_1} and \eqref{eq_rho_2}, we have 
    \begin{equation}\label{eq_rho_3}
        \begin{split}
            \{ a, h_0 \}
            &\leq \frac{c^{-2}f^\prime}{f}\underbrace{(2\rho-\sqrt{2h_0}-\rho)a}_{=-f^{-2}a^2\leq 0}+ Cff^\prime r^{-1-\mu} (1+h_0)\\
            &\leq Cff^\prime r^{-1-\mu} (1+h_0)
            =\frac{C(1+h_0)}{2c^{-2}\rho}\{ f^2 r^{-1-\mu}, h_0\}+(1+\mu)f^2 r^{-2-\mu}.  
        \end{split}
    \end{equation}
    In the following we only consider $(r, \theta, \rho, \eta)$ on the energy surface $\{h_0=E_0\}$ such that $\rho$ is sufficiently close to $\rho_\infty=\sqrt{2E_0}$. Then \eqref{eq_rho_3} becomes 
    \begin{equation}
        \label{eq_rho_4}
        \{ a, h_0 \}\leq C\{ f^2 r^{-1-\mu}, h_0\}+Cf(r)^2 r^{-2-\mu}. 
    \end{equation}
    Fix a large $R>0$. Since $r\mapsto f(r)/r^{c_0}$ is monotonically increasing by \eqref{eq_ineq_f_logbdd} and 
    \[
        \frac{\diff}{\diff r}\left( \frac{f(r)}{r^{c_0}}\right)=\frac{f^\prime - c_0 f r^{-1}}{r^{c_0}}\geq 0, 
    \]
    we have the inequality $f(r)/f(R)\leq (r/R)^{c_0}$ for $R\geq r$. Thus \eqref{eq_rho_4} becomes
    \[
        \{ a, h_0 \}\leq C\{ f^2 r^{-1-\mu}, h_0\}+Cf(R)^2 R^{-c_0}r^{-2-\mu+c_0}. 
    \]

    Now we substitute $(r, \theta, \rho, \eta)=(r(t), \theta (t), \rho (t), \eta(t))$. Then we have 
    \begin{align*}
        \frac{\diff}{\diff t}(a(r(t), \rho(t)))
        &\leq C\frac{\diff}{\diff t}(f(r(t))^2 +r(t)^{-1-\mu})+Cf(R)^2 R^{-c_0} r(t)^{-2-\mu+c_0}
    \end{align*}
    for large $t$ such that $r(t)\leq R$. Note that $t\mapsto r(t)$ is monotonically increasing for large $t$ by $\rho_\infty>0$ and the Hamilton equation \eqref{eq_hamilton_equation_radial}. Integrating both sides in $[T, t]$ and substituting $R=r(t)$, we obtain 
    \[
        a(r(t), \rho (t))\leq Cf(r(t))^2 r(t)^{-1-\mu}+Cf(r(t))^2 r(t)^{-c_0}\int_T^t r(s)^{-2-\mu+c_0}\, \diff s +C.  
    \]
    Recall $a(r, \rho)=f(r)^2(\sqrt{2h_0}-\rho)$. Then we have 
    \begin{equation}
        \label{eq_rho1_duhamel}
        \sqrt{2E_0}-\rho (t)\leq Cr(t)^{-1-\mu}C r(t)^{-c_0}\int_T^t r(s)^{-2-\mu+c_0}\, \diff s +Cf(r(t))^{-2}. 
    \end{equation}
            
    Since $f(r)\geq C^{-1}r^{c_0}$ by \eqref{eq_ineq_f_logbdd} and 
    \begin{align*}
        &r(t)^{-c_0}\int_T^t r(s)^{-2-\mu+c_0}\, \diff s \\
        &=r(t)^{-c_0}\int_T^t \frac{r(s)^{-2-\mu+c_0}\diff r/ds(s)}{c(r(s), \theta (s))^{-2}\rho (s)}\, \diff s
        \leq Cr(t)^{-c_0}\int_T^t r(s)^{-2-\mu+c_0}\frac{\diff r}{ds}(s)\, \diff s \\
        &\leq 
        \begin{cases}
            Cr(t)^{-1-\mu}+Cr(t)^{-c_0} & \text{if } \mu+1\neq c_0, \\
            Cr(t)^{-1-\mu}\log r(t) & \text{if } \mu+1=c_0
        \end{cases}
        \\
        &\leq Cr(t)^{-1-\delta}
    \end{align*}
    for $0<\delta<\min \{ \mu, 2c_0-1\}$, \eqref{eq_rho1_duhamel} becomes 
    \[
        \sqrt{2E_0}-\rho (t)\leq Cr(t)^{-1-\mu}+Cr(t)^{-1-\delta}+Cr(t)^{-2c_0}\leq Cr(t)^{-1-\delta}. 
    \] 
    We can replace $r(t)$ to $t$ since 
    \begin{align*}
        &\lim_{t\to\infty}\frac{r(t)}{t} \\
        &=\lim_{t\to\infty} \frac{t-T}{t}\int_0^1 c(r(T+(t-T)s), \theta (T+(t-T)s))^{-2}\rho (T+(t-T)s)\, \diff s \\
        &=\rho_\infty
    \end{align*}
    by the Hamilton equation \eqref{eq_hamilton_equation_radial} and the Lebesgue dominated convergence theorem. Thus $\sqrt{2E_0}-\rho (t)\leq O(t^{-1-\delta})$. 

    The converse inequality is easily proved by the estimate 
    \[
        \rho(t)\leq c\sqrt{2E_0}\leq (1+Cr(t)^{-1-\mu})\sqrt{2E_0}\leq \sqrt{2E_0}+O(t^{-1-\mu}). \qedhere
    \]
\end{proof}

\begin{proof}[Proof of Theorem \ref{theo_classical_estimate}]
    We already proved the existence of the asymptotically radial momentum $\rho_\infty=\lim_{t\to\infty}\rho(t)$ in Lemma \ref{lemm_asymp_radial}. Noting the integrability of $t^{-1-\delta}$ in $[1, \infty)$ and integrating both sides of \eqref{eq_lemm_rho1}, we obtain
    \[
        \rho_\infty t-C \leq r(t)\leq \rho_\infty+C. 
    \]
    
    In the following we prove the existence of asymptotic angle and angular momentum $(\theta_\infty, \eta_\infty):=\lim_{t\to\infty}(\theta (t), \eta(t))\in T^*S$. Let $d_S$ be the distance function on $S$ associated with the Riemannian metric $h(1, \theta, \diff \theta)$. For $t\geq s$, we have 
        \begin{equation}\label{eq_theta_cauchy} d_S(\theta(t), \theta(s))\leq \int_s^t \left|\frac{\diff \theta}{\diff t}(\sigma)\right|_{h(1)}\, \diff \sigma\leq \int_s^t f(r(\sigma))^{-1}\left|\frac{\diff \theta}{\diff t}(\sigma)\right|_{h(r(\sigma))}\, \diff \sigma. 
        \end{equation}
        The angle component of the Hamilton equation is 
        \begin{equation}\label{eq_hamilton_equation_angle}
            \frac{\diff \theta_j}{\diff t}(t)=\sum_{k=1}^{n-1} h^{jk}(r(t), \theta (t))\eta_k (t), 
        \end{equation}
        which implies 
        \[
            \left|\frac{\diff \theta}{\diff t}(t)\right|_{h(r)}=|\eta|_{h^*(r)}. 
        \]
        By this relation and the energy conservation law, we have 
            \[
                \left|\frac{\diff \theta}{\diff t}(\sigma)\right|_{h(r(\sigma))}=|\eta(\sigma)|_{h^*(r(\sigma))}=\sqrt{2E_0-c^{-2}\rho^2} 
            \]
        by \eqref{eq_ineq_model_bdd}. We apply Lemma \ref{lemm_rho1} and obtain 
            \[\sqrt{2E_0-c^{-2}\rho^2}\leq \sqrt{Ct^{-1-\delta}}=Ct^{-(1+\delta)/2}. 
            \]
            Combining this with $f(r)\geq C^{-1}r^{c_0}$, which follows from \eqref{eq_ineq_f_logbdd}, \eqref{eq_theta_cauchy} becomes 
            \[
                d_S(\theta(t), \theta(s))\leq C\int_s^t r(\sigma)^{-c_0}\sigma^{-(1+\delta)/2}\, \diff \sigma \leq C\int_s^t \sigma^{-c_0-(1+\delta)/2}\, \diff \sigma. 
            \]
            Since $c_0+(1+\delta)/2>1$ by $c_0>1/2$, the integrand $\sigma^{-c_0^(1+\delta)/2}$ is integrable in $[1, \infty)$. Thus 
            \[
                d_S(\theta(t), \theta(s))\to 0 \quad (t\geq s \to \infty). 
            \]
            Hence the limit 
            \[
                \theta_\infty:=\lim_{t\to\infty}\theta (t)
            \]
            exists by the completeness of compact Riemannian manifolds.

            Finally we consider the $\eta$ component. Take local coordinates near $\theta_\infty$. We take sufficiently large $T>0$ such that $\theta(t)$ is in the coordinate neighborhood for all $t\geq T$. In the associated canonical coordinates, we have 
            \begin{align*}\left|\frac{\diff \eta_i}{\diff t}\right|
                &\leq Cr^{-1-\mu}+Ch^*(r, \theta, \eta)= Cr^{-1-\mu}+C(2E_0-c^{-2}\rho^2) \\
                &\leq Ct^{-1-\mu}+C(2E_0+Ct^{-1-\mu}-(\rho_\infty^2-Ct^{-1-\delta}))=Ct^{-1-\delta}. 
            \end{align*}
            Thus
            \[ |\eta_i(t)-\eta_i(s)|\leq C\int_s^t \sigma^{-1-\delta}\, \diff \sigma=C(s^{-\delta}-t^{-\delta})\to 0 \quad (t\geq s \to \infty)\]
            for $t\geq s\geq T$. Hence the limit $\eta_{\infty, i}:=\lim_{t\to\infty}\eta_i(t)$ exists. We pull $(\eta_{\infty, 1}, \ldots, \eta_{\infty, n-1})$ back by the canonical coordinates and obtain the asymptotic angular momentum $\eta_\infty\in T^*_{\theta_\infty}S$. 
\end{proof}

\section{Pseudodifferential operators on manifolds}\label{sect_psido_manifolds}

\subsection{Symbol classes}\label{subs_cylindrical_class}

We first introduce a suitable symbol class for analyzing the symbols defined as \eqref{eq_defi_op_at}. In order to deduce global properties of pseudodifferential operators ($L^2$ boundedness for example), we need to control behavior of symbols near infinity ($r\to \infty$). 

\begin{defi}
    Let $m\in \mathbb{R}$. A function $a\in C^\infty (T^*M)$ belongs to $S_\mathrm{cyl}^m(T^*M)$ if it satisfies the following conditions. 
    \begin{itemize}
        \item For every local coordinate $\varphi: U\to V$, the push forward $\tilde\varphi_*a$ by the induced canonical coordinate $\tilde\varphi: T^*U\to T^*V$ belongs to $S^m_\mathrm{loc}(T^*V)$. Here $S^m_\mathrm{loc}(T^*V)\subset C^\infty (V\times \mathbb{R}^n)$ stands for the set of all functions which satisfy 
        \[
            \jbracket{\xi}^{-m+|\beta|}\partial_x^\alpha \partial_\xi^\beta a\in L^\infty (K\times \mathbb{R}^n)
        \]
        for all compact subsets $K\subset V$ and all multiindices $\alpha, \beta\in \mathbb{Z}_{\geq 0}^n$. 
        \item For any polar coordinate $\varphi: U\to V$ in the end $E$, the push forward $\tilde\varphi_*a(r, \theta, \rho, \eta)$ by the induced canonical coordinate $\tilde\varphi: T^*U\to T^*V$ satisfies 
        \[
            \jbracket{\xi}^{-m+|\beta|}\partial_{r, \theta}^\alpha \partial_\xi^\beta (\tilde \varphi_* a) \in L^\infty ([\delta, \infty)\times K^\prime\times \mathbb{R}^n)
        \]
        for all $\delta>0$, compact subsets $K^\prime\subset \mathbb{R}^{n-1}$ with $\mathbb{R}_+\times K^\prime \subset V$, and all multiindices $\alpha, \beta\in \mathbb{Z}_{\geq 0}^n$. Here $\xi=(\rho, \eta)$. 
    \end{itemize}
\end{defi}

\begin{rema*}
    We denote the subscript ``$\mathrm{cyl}$'' in $S^m_\mathrm{cyl}(T^*M)$ since the symbol in $S^m_\mathrm{cyl}(T^*M)$ can be regarded as a natural symbol class on $E \simeq \mathbb{R}_+\times S$ with a cylindrical metric $\diff r^2+\diff \theta^2$. 
\end{rema*}

It is useful to introduce a terminology for describing supports of symbols up to $O(\hbar^\infty)$. We define it following \cite{Ito06}. 

\begin{defi}\label{defi_support_modulo}
    An $\hbar$-dependent symbol $a(\hbar; x, \xi)\in S^m_\mathrm{cyl}(T^*M)$ satisfies $\rmop{supp}a\subset K$ modulo $O(\hbar^\infty)$ if there exists a possibly $\hbar$-dependent symbol $\tilde a(\hbar; x, \xi)\in S^m_\mathrm{cyl}(T^*M)$ such that $\rmop{supp} \tilde a\subset K$ and $a-\tilde a=O_{S^0_\mathrm{cyl}}(\hbar^\infty)$. 
\end{defi}

We explain a motivation to consider the support modulo $O(\hbar^\infty)$ and also recall facts on a symbol calculus on Euclidean spaces. Let $S^m(T^*\mathbb{R}^n)$ be the Kohn-Nirenberg symbol class 
\[
    S^m(T^*\mathbb{R}^n):=\{ a\in C^\infty (T^*\mathbb{R}^n) \mid |\partial_x^\alpha \partial_\xi^\beta a(x, \xi)|\leq C_{\alpha\beta}\jbracket{\xi}^{m-|\beta|}\}. 
\]
Similarly to Definition \ref{defi_support_modulo}, we call $\rmop{supp}a\subset K$ modulo $O(\hbar^\infty)$ for $a(\hbar; x, \xi)\in S^m(T^*\mathbb{R}^n)$ if there exists a symbol $\tilde a(\hbar; x, \xi)$ such that $\rmop{supp}\tilde a\subset K$ and $a-\tilde a=O_{S^0(T^*\mathbb{R}^n)}(\hbar^\infty)$. 
What to recall are the composition of symbols and the changing variables. 

\fstep{Composition of symbols}For $a\in S^{m_1}(T^*\mathbb{R}^n)$ and $b\in S^{m_2}(T^*\mathbb{R}^n)$, we can calculate a symbol $(a\# b)(\hbar; x, \xi)\in S^{m_1+m_2}(T^*\mathbb{R}^n)$ such that
\[
    a^\mathrm{w}(x, \hbar D)b^\mathrm{w}(x, \hbar D)=(a\# b)^\mathrm{w}(x, \hbar D)
\] 
and 
\begin{align*}
    &(a\# b)(\hbar; x, \xi) \\
    &=\sum_{k=0}^N \frac{\hbar^k}{k!}\left(\frac{D_x\cdot D_{\xi^\prime}-D_{x^\prime} \cdot D_\xi}{2i}\right)^k (a(x, \xi)b(x^\prime, \xi^\prime))\biggr|_{\substack{x^\prime=x \\ \xi^\prime=\xi}} \\
    &\quad +O_{S^{m_1+m_2-N-1}(T^*\mathbb{R}^n)}(\hbar^{N+1})
\end{align*}
for any integer $N\geq 0$ (see Theorem 9.5 in \cite{Zworski12}). Each term 
\[
    \left(\frac{D_x\cdot D_{\xi^\prime}-D_{x^\prime} \cdot D_\xi}{2i}\right)^k (a(x, \xi)b(x^\prime, \xi^\prime))\biggr|_{\substack{x^\prime=x \\ \xi^\prime=\xi}}
\]
is supported in $\rmop{supp}(ab)$. Thus if we define a symbol $c\in S^{m_1+m_2}(T^*\mathbb{R}^n)$ as an asymptotic sum
\[
    c(\hbar; x, \xi)\sim \sum_{k=0}^\infty \frac{\hbar^k}{k!}\left(\frac{D_x\cdot D_{\xi^\prime}-D_{x^\prime} \cdot D_\xi}{2i}\right)^k (a(x, \xi)b(x^\prime, \xi^\prime))\biggr|_{\substack{x^\prime=x \\ \xi^\prime=\xi}}
\]
by the Borel theorem, then $\rmop{supp}c \subset \rmop{supp}(ab)$ and $a\# b-c=O_{S^0(T^*\mathbb{R}^n)}(\hbar^\infty)$. Hence we have $\rmop{supp}(a\# b)\subset \rmop{supp} (ab)$ modulo $O(\hbar^\infty)$. 

\fstep{Changing variables}For a suitable diffeomorphism $\gamma: \mathbb{R}^n\to \mathbb{R}^n$ and a symbol $a\in S^m(T^*\mathbb{R}^n)$, we have 
\[
    \gamma_* a^\mathrm{w}(x, \hbar D)\gamma^*=\tilde a^\mathrm{w}(x, \hbar D)
\]
for some symbol $\tilde a(\hbar; x, \xi)\in S^m(T^*\mathbb{R}^n)$ which has an asymptotic expansion 
\[
    \tilde a(\hbar; x, \xi)\sim \sum_{j=0}^\infty \hbar^j \tilde a_j(x, \xi), \quad \tilde a_j\in S^{m-j}(T^*\mathbb{R}^n)
\]
and 
\[
    \tilde a_0(x, \xi)=a(\gamma (x), \diff \gamma (x)^{-1*}\xi)=:\tilde\gamma^*a(x, \xi)
\]
and $\rmop{supp}a_j \subset \rmop{supp}a_0$ for all $j\geq 0$. This implies $\rmop{supp}\tilde a\subset \rmop{supp}\tilde\gamma^*a$ modulo $O(\hbar^\infty)$. Here $\diff \gamma (x)^{-1*}\xi=\eta$ if and only if 
\[
    \xi_j=\sum_{k=1}^n \frac{\partial\gamma_k}{\partial x_j}(x)\eta_k. 
\]
Furthermore, since we consider pseudodifferential operators acting on half-densities (we will explain them in Section \ref{subs_psido_half_densities}), we have $\tilde a_1=0$. Thus 
\[
    \tilde a=\tilde\gamma^*a+O_{S^{m-2}(T^*\mathbb{R}^n)}(\hbar^2). 
\]
For more details, see Theorem 9.9 and Theorem 9.10 in \cite{Zworski12} or Proposition E.10 in \cite{Dyatlov-Zworski19}. 

\subsection{Pseudodifferential operators acting on half-densities}\label{subs_psido_half_densities}
Before definition of pseudodifferential operators, we recall basic facts on half-densities on manifolds. For a manifold $M$, a line bundle $\pi: \Omega^{1/2}(M)\to M$ is defined as follows and call sections of the line bundle $\Omega^{1/2}(M)\to M$ half-densities of $M$:  

\fstep{Fiber}A fiber $\pi^{-1}(x)$ is a complex vector space spanned by functions of the form 
\[
    |\omega|^{1/2}: (v_1, \ldots, v_n)\in (T^*M)^n\longmapsto |\omega (v_1, \ldots, v_n)|^{1/2}\in \mathbb{R}, \quad \omega\in \Lambda^n T^*_xM. 
\]

\fstep{Local trivialization}Each local coordinates $\varphi=(x_1, \ldots, x_n): U\to V$ on $M$ induces a local line bundle isomorphism 
\[
    (x, v|\diff x_1 \wedge \cdots \wedge \diff x_n|^{1/2})\in \pi^{-1}(U)\overset{\simeq}{\longmapsto} (x, v)\in U\times \mathbb{C}.  
\]

\fstep{}We denote the space of all smooth compactly supported half-densities by $C_c^\infty (M; \Omega^{1/2})$. 

We employ two manipulations for half-densities. 

\fstep{Inner product}Similarly to the definition of integration of differential forms, we define 
\begin{equation}\label{eq_inner_half_densities}
    \jbracket{u, v}:=\sum_{\iota\in I}\int_{\mathbb{R}^n} ((\kappa_\iota \tilde u \overline{\tilde v})\circ \varphi_\iota^{-1})(x)\, \diff x. 
\end{equation}
Here 
\begin{itemize}
    \item $\{\varphi_\iota: U_\iota\to V_\iota\}_{\iota\in I}$ is a locally finite atlas; 
    \item $\{\kappa_\iota\in C^\infty(M)\}_{\iota\in I}$ is a partition of unity subordinate to $\{ U_\iota \}_{\iota\in I}$; 
    \item $u=\tilde u |\diff x_1\wedge \cdots \wedge \diff x_n|^{1/2}$ and $v=\tilde v |\diff x_1\wedge \cdots \wedge \diff x_n|^{1/2}$ in $U_\iota$ are compactly supported half-densities. 
\end{itemize}
\eqref{eq_inner_half_densities} is independent of the choice of an atlas and a partition of unity. 
The inner product \eqref{eq_inner_half_densities} induces an $L^2$-norm $\|u\|_{L^2}:=\jbracket{u, u}^{1/2}$. We define $L^2(M; \Omega^{1/2})$ as the completion of $C_c^\infty (M; \Omega^{1/2})$ with respect to the norm $\|\cdot \|_{L^2}$. 

\fstep{Pull back}For a smooth map $f: M\to M$, we define a pull back $f^*u$ for $u\in C_c^\infty (M; \Omega^{1/2})$ as 
\[
    f^*u (x)(v_1, \ldots, v_n):=u(df(x)(v_1), \ldots, df (x)(v_n))
\]
for $x\in M$ and $v_1, \ldots , v_n \in T_xM$. If $f(x_1, \ldots, x_n)=(y_1, \ldots, y_n)$ locally, then 
\[
    f^* (\tilde u |\diff y|^{1/2})=(\tilde u\circ f)(x)\left|\det \frac{\partial f}{\partial x}(x)\right|^{1/2}|\diff x|^{1/2}. 
\]

All pull back manipulations by diffeomorphisms are unitary operators with respect to the inner product \eqref{eq_inner_half_densities}. 

If $f: M\to M$ be a diffeomorphism, then we define a push forward of half-densities as $f_*:=(f^{-1})^*$. 

\subsection{Properties of pseudodifferential operators}

For the definition of a quantization procedure, we take a finite atlas $\{ \varphi_\iota: U_\iota \to V_\iota \}$ on $M$ as below. 
\begin{enumerate}
    \item We cover the compact subset $r^{-1}(0)$ by finite atlas $\{\varphi_\iota: U_\iota \to V_\iota \}_{\iota\in I_K}$. Here $U_\iota\subset M$, $V_\iota\subset \mathbb{R}^n$ and $\# I_K<\infty$. 
    \item We cover the compact manifold $S$ by finite atlas $\{\varphi^\prime_\iota: U^\prime_\iota \to V^\prime_\iota \}_{\iota\in I_\infty}$. Here $U^\prime_\iota\subset S$, $V^\prime_\iota\subset \mathbb{R}^{n-1}$ and $\# I_\infty<\infty$. We set $U_\iota:=\mathbb{R}_+\times U^\prime_\iota$, $V_\iota:=\mathbb{R}_+\times V^\prime_\iota$ and $\varphi_\iota:=\rmop{id}\times \varphi^\prime_\iota: U_\iota \to V_\iota$. 
    \item We define $I:=I_K \cup I_\infty$ (assuming that $I_K \cap I_\infty =\varnothing$). 
\end{enumerate}

Furthermore, we take a partition of unity $\{\kappa_\iota\in C^\infty (M)\}_{\iota\in I}$ subordinate to $\{\varphi_\iota \}_{\iota\in I}$ such that the following statements hold. 
\begin{itemize}
    \item For $\iota\in I_K$, $\kappa_\iota\in C_c^\infty (U_\iota)$. 
    \item For $\iota\in I_\infty$, $\kappa_\iota$ is a cylindrical function (see Definition \ref{defi_conical}) with $\rmop{supp}\kappa_\iota \subset U_\iota$. 
    \item $\sum_{\iota\in I} \kappa_\iota=1$. 
\end{itemize}

In the following we fix the atlas $\{\varphi_\iota\}_{\iota\in I}$ and the partition of unity $\{\kappa_\iota\}_{\iota\in I}$. 

\begin{defi}[(Non-canonical) quantization]\label{defi_quantization}
    We fix cylindrical functions $\chi_\iota\in C^\infty(M)$ with $\rmop{supp} \chi_\iota \subset U_\iota$ and $\rmop{supp} \chi_\iota =1$ near $\rmop{supp}\kappa_\iota$. For a symbol $a\in S^m_\mathrm{cyl}(T^*M)$ and a function $u\in C_c^\infty (M; \Omega^{1/2})$, we define a quantization as 
    \begin{equation}\label{eq_psido_cyl} \Oph (a)u:=\sum_{\iota\in I}\chi_\iota \varphi_\iota^*(\tilde\varphi_{\iota*}(\kappa_\iota a))^\mathrm{w}(x, \hbar D)\varphi_{\iota*}(\chi_\iota u).  
    \end{equation}

    Here pseudodifferential operators acting on half-densities on Euclidean spaces are defined as 
    \[
        a^\mathrm{w}(x, \hbar D)(\tilde u |\diff x|^{1/2}):=\frac{1}{(2\pi\hbar)^n}\int_{\mathbb{R}^{2n}} a\left(\frac{x+y}{2}, \xi\right) e^{i\xi \cdot (x-y)/\hbar}\tilde u(y)\, \diff y\diff \xi |\diff x|^{1/2}. 
    \]
\end{defi}

Examples which we keep in mind are the quantization of polynomials in momentum variables. In polar coordinates, the quantization of polynomials in $S^m_\mathrm{cyl}(T^*M)$ is a sum of $a_\beta (r, \theta)D_r^{\beta_0}D_\theta^{\beta^\prime}$ with bounded $a_\alpha(r, \theta)$ in the sense that 
\[
    |\partial_r^{\alpha_0}\partial_\theta^{\alpha^\prime}a_\beta (r, \theta)|\leq C_{\alpha\beta}. 
\]

\begin{rema*}
    If $(M, g)$ is a Riemannian manifold and $\varphi: U\to V$ are local coordinates, then $\varphi^* a^\mathrm{w}(x, \hbar D)\varphi_*$ is 
    \begin{align*}
        &\varphi^*a^\mathrm{w}(x, \hbar D)\varphi_*(\tilde u |\vol_g|^{1/2}) \\
        &:=\frac{g(x)^{-1/4}}{(2\pi\hbar)^n}\int_{\mathbb{R}^{2n}} a\left(\frac{x+y}{2}, \xi\right) e^{i\xi \cdot (x-y)/\hbar}\tilde u(y)g(y)^{1/4}\, \diff y\diff \xi |\vol_g|^{1/2}. 
    \end{align*}
    Here $g(x)$ is defined by the relation $\vol_g(x)=g(x)^{1/2}\diff x$. The difference between pseudodifferential operators acting on half-densities and those acting on functions is the existence of the factor $g(x)^{-1/4}g(y)^{1/4}$. 
\end{rema*}

We employ composition and commutator of pseudodifferential operators, boundedness on the $L^2$ space and the sharp G\aa rding inequality in proof of the main theorem.

\begin{theo}\label{theo_psido_composition}
    Let $m_1, m_2\in \mathbb{R}$. For $a\in S^{m_1}_\mathrm{cyl}(T^*M)$, $b\in S^{m_2}_\mathrm{cyl}(T^*M)$, the following statements hold. 
    \begin{enumerate}
        \item The composition $\Oph (a)\Oph (b)$ is represented by some symbol $c(\hbar; x, \xi)\in S^{m_1+m_2-2}_\mathrm{cyl}(T^*M)$ as 
        \begin{equation}\label{eq_composition}
            \Oph (a) \Oph (b)=\Oph \left(ab+\frac{i\hbar}{2}\{a, b\}+\hbar^2 c\right)+O_{L^2\to L^2}(\hbar^\infty). 
        \end{equation}
        The symbol $c(\hbar; x, \xi)$ satisfies $\rmop{supp}c(\hbar; x, \xi)\subset \rmop{supp}(ab)$ mod $O(\hbar^\infty)$. 
        \item In particular, the commutator $[\Oph (a), \Oph (b)]$ is represented by some $c\in S^{m_1+m_2-2}_\mathrm{cyl}(T^*M)$ as 
        \begin{equation}\label{eq_commutator}
            [\Oph (a), \Oph (b)]=\Oph (i\hbar\{a, b\}+\hbar^2 c)+O_{L^2\to L^2}(\hbar^\infty). 
        \end{equation}
        The symbol $c=c (\hbar; x, \xi)$ satisfies $\rmop{supp}c(\hbar; x, \xi)\subset \rmop{supp}(ab)$ mod $O(\hbar^\infty)$. 
    \end{enumerate}
\end{theo}

\begin{rema*}
    One can prove that the $O_{L^2\to L^2}(\hbar^\infty)$ term in \eqref{eq_composition} and \eqref{eq_commutator} has smooth integral kernels. However, we do not use it in this paper. 
\end{rema*}

\begin{theo}
    \label{theo_L2_bdd_cyl}
    For any symbol $a\in S^0_\mathrm{cyl}(T^*M)$, the operator $\Oph (a)$ is bounded on $L^2(M; \Omega)$. 

    Furthermore, if the symbol $a$ also depends on some parameter $\tau\in \Omega$ and are uniformly bounded in $S^0_\mathrm{cyl}(T^*M)$, then the operator norm $\|\Oph (a)\|_{L^2\to L^2}$ is uniformly bounded with respect to $\tau\in\Omega$. 
\end{theo}

\begin{proof}
    Each terms $\chi_\iota \varphi_\iota^*(\tilde\varphi_{\iota*}(\kappa_\iota a))^\mathrm{w}(x, \hbar D)\varphi_{\iota*}\chi_\iota$ in the definition \eqref{eq_psido_cyl} of pseudodifferential operators is bounded on $L^2(M; \Omega)$ by the Calder\'{o}n-Vaillancourt theorem for $\tilde\varphi_{\iota*}(\kappa_\iota a)\in S^0 (T^*\mathbb{R}^n)$ and the unitarity of the pull back $\varphi^*$ and the push forward $\varphi_*$. Thus the finite sum \eqref{eq_psido_cyl} over $\iota\in I$ is also a bounded operator on $L^2(M; \Omega^{1/2})$. 
\end{proof}

\begin{theo}[Sharp G\aa rding inequality]\label{theo_sharp_garding}
    For every $a\in S^0_\mathrm{cyl}(T^*M)$ with $\rmop{Re} a\geq 0$, there exists a real symbol $b(\hbar; x, \xi)\in S^0_\mathrm{cyl}(T^*M)$ such that the inequality 
        \begin{equation}\label{eq_sharp_garding}
            \rmop{Re} \Oph (a)\geq -\hbar \Oph (b)+O_{L^2\to L^2}(\hbar^\infty)
        \end{equation}
    holds and $\rmop{supp}b \subset \rmop{supp} a$ mod $O(\hbar^\infty)$. 
    
    Furthermore, if the symbol $a$ also depends on some parameter $\tau\in \Omega$ and are uniformly bounded in $S^0_\mathrm{cyl}(T^*M)$, then the $O_{L^2\to L^2}(\hbar^\infty)$ in \eqref{eq_sharp_garding}, and the symbol $b\in S^0_\mathrm{cyl}(T^*M)$ itself are uniformly bounded with respect to $\tau\in\Omega$. 
\end{theo}
We prove Theorem \ref{theo_psido_composition} and Theorem \ref{theo_sharp_garding} in Subsection \ref{subs_proof_psido}. 

In order to treat $[A_\hbar (t), H]$ in \eqref{eq_positive_heisenberg_derivative}, we represent the semiclassical Laplacian
\[
    -\hbar^2\triangle_g (u|\vol_g|^{1/2}):=\rmop{div}(\grad u)|\vol_g|^{1/2}
\]
as a pseudodifferential operator: 

\begin{theo}\label{theo_laplacian_psido}
    We have 
    \[
        -\hbar^2\triangle_g=\Oph (|\xi|_{g^*}^2+\hbar^2 V_g), 
    \]
    where $V_g\in C^\infty (M)$ is defined as 
    \[
        V_g(x):=\sum_{\iota\in I}\sum_{j, k=1}^n\left(\frac{1}{4}\partial_{x_j}\partial_{x_k}(\kappa_\iota g_\iota^{jk})+g_\iota^{-1/4}\partial_{x_j}(\kappa_\iota g_\iota^{jk}\partial_{x_k}g_\iota^{1/4})\right), 
    \]
    $g_\iota^{jk}$, $g_\iota$ are defined on $U_\iota$ as 
    \[
        (g_\iota^{jk}):=(g^\iota_{jk})^{-1}, \quad 
        g_\iota:=\sqrt{\det (g^\iota_{jk})}, \quad 
        \text{where } g=\sum_{j, k=1}^n g^\iota_{jk}\diff x_j \diff x_k. 
    \]
    Furthermore, $V_g$ belongs to the symbol class $S^0_\mathrm{cyl}(T^*M)$. 
\end{theo}

\begin{proof}
    We decompose $-\hbar^2\triangle_g$ into 
    \begin{equation}\label{eq_laplacian_decomposition}
        -\hbar^2 \triangle_g u=\sum_{\iota\in I} \chi_\iota \rmop{div} (\kappa_\iota \rmop{grad} (\chi_\iota u)). 
    \end{equation}
        
    Direct calculation of $(\tilde \varphi_{\iota*}(\kappa_\iota |\xi|_{g^*}^2))^\mathrm{w}(x, \hbar D)$ shows that 
    \[
        \varphi_\iota^*(\tilde \varphi_{\iota*}(\kappa_\iota |\xi|_{g^*}^2))^\mathrm{w}(x, \hbar D)\varphi_{\iota *}u=-\rmop{div} (\kappa_\iota \rmop{grad}u)-V_\iota u
    \]
    and 
    \begin{equation}\label{eq_geometric_potential_local}
        V_\iota(x):=\frac{1}{4}\partial_{x_j}\partial_{x_k}(\kappa_\iota g_\iota^{jk})
        +g_\iota^{-1/4}\partial_{x_j}(\kappa_\iota g_\iota^{jk}\partial_{x_k}g_\iota^{1/4}). 
    \end{equation}
    Thus, by \eqref{eq_laplacian_decomposition}, we have
    \[
        -\hbar^2 \triangle_g u=\sum_{\iota \in I} \chi_\iota \varphi_\iota^*(\tilde \varphi_{\iota*}(\kappa_\iota |\xi|_{g^*}^2))^\mathrm{w}(x, \hbar D)\varphi_{\iota *}(\chi_\iota u)
        +\sum_{\iota \in I} V_\iota \chi_\iota^2 u. 
    \]
    $V_\iota \chi_\iota^2$ by $\chi_\iota =1$ on $\rmop{supp}V_\iota$ and $V_g=\sum_{\iota\in I}V_\iota$ implies that 
    \[
        \sum_{\iota \in I} V_\iota \chi_\iota^2 u=Vu=\sum_{\iota \in I}\chi_\iota \varphi_\iota^* (\tilde\varphi_{\iota*}(\kappa_\iota V))^\mathrm{w}\varphi_{\iota*}(\chi_\iota u)=\Oph (V)u. 
    \]

    We have to show that $V_g\in S^0_\mathrm{cyl}(T^*M)$. The problem is the behavior of derivatives of $g^{jk}_\iota$ and $g_\iota$ for $\iota\in I_\infty$. By \eqref{eq_metric_polar}, we have 
    \begin{align*}
        g^{jk}_\iota&=
        \begin{cases}
            c(r, \theta)^{-2} & \text{if } j=k=1, \\
            h^{j-1, k-1}_\iota (r, \theta) & \text{if } j, k\geq 2, \\
            0 & \text{otherwise}, 
        \end{cases} \\ 
        g_\iota &=c(r, \theta)^2 h_\iota (r, \theta), 
    \end{align*}
    where 
    \[
        (h^{jk}_\iota):=(h^\iota_{jk})^{-1}, \quad h_\iota:=\det (h^\iota_{jk}), \quad h(r, \theta, \diff \theta)=\sum_{j, k=1}^{n-1} h^\iota_{jk}(r, \theta)\diff \theta_j \diff \theta_k. 
    \]
    Assumption \ref{assu_classical} \ref{assu_sub_short_range} (in particular $c\to 1$ as $r\to \infty$) and Assumption \ref{assu_higher_derivative} implies the boundedness of 
    \[
        |\partial_{r, \theta}^\alpha \partial_{x_j}\partial_{x_k}(\kappa_\iota g_\iota^{jk})|
    \]
    and 
    \[
        |\partial_{r, \theta}^\alpha (g_\iota^{-1/4}\partial_{x_j}(\kappa_\iota g_\iota^{jk}\partial_{x_k}g_\iota^{1/4}))|
    \]
    in \eqref{eq_geometric_potential_local}. This shows that $V_\iota\in S^0_\mathrm{cyl}(T^*M)$ for $\iota\in I_\infty$. 
\end{proof}

\begin{rema*}
    $V_g(x)$ depends on choices of atlas on $M$. 
\end{rema*}

\subsection{Proof of Theorem \ref{theo_psido_composition} and Theorem \ref{theo_sharp_garding}}\label{subs_proof_psido}

It is useful to introduce a notation of pseudodifferential operators associated with locally defined symbols. 

\begin{defi}\label{defi_local_psido}
    For $a_\iota\in S^m_\mathrm{cyl}(T^*\mathbb{R}^n)$ and $u\in C_c^\infty (M; \Omega^{1/2})$, we define
    \[
        \Ophloc (a_\iota)u:=\chi_\iota \varphi_\iota^* a_\iota^\mathrm{w}(x, \hbar D)\varphi_{\iota*}(\chi_\iota u). 
    \]
\end{defi}

The operators in Definition \ref{defi_local_psido} are represented by a quantization of globally defined symbols.

\begin{lemm}\label{lemm_psido_local_global}
    Assume that symbols $a_\iota\in S^m(T^*\mathbb{R}^n)$ satisfy $\rmop{supp}a_\iota \subset \pi^{-1}(\rmop{supp} \varphi_{\iota*} \kappa_\iota)$ for all $\iota\in I$, where $\pi: T^*M \to M$ is the natural projection. Then there exists $a(\hbar; x, \xi)\in S^m_\mathrm{cyl}(T^*M)$ such that 
    \begin{equation}\label{eq_psido_local_global}
        \sum_{\iota\in I}\Ophloc (a_\iota)=\Oph (a)+O_{L^2\to L^2}(\hbar^\infty). 
    \end{equation}
    This symbol $a\in S^m_\mathrm{cyl}(T^*M)$ satisfies $\rmop{supp}a\subset \rmop{supp}a_0$ modulo $O(\hbar^\infty)$ where 
    \[
        a_0(x, \xi)=\sum_{\iota\in I} \tilde \varphi_\iota^* a_\iota (x, \xi). 
    \]

    Furthermore, if the symbols $a_\iota$ also depend on some parameter $\tau\in \Omega$ and are uniformly bounded in $S^m(T^*\mathbb{R}^n)$, then the $O_{L^2\to L^2}(\hbar^\infty)$ in \eqref{eq_psido_local_global} are uniformly bounded with respect to $\tau\in\Omega$. 
\end{lemm}

\begin{proof}
    The explicit form of $\Oph (a_0)$ is
    \begin{align}
        \Oph (a_0)
        &=\sum_{U_\iota \cap U_{\iota^\prime}\neq \varnothing} 
        \chi_\iota \varphi_\iota^* (\tilde\varphi_{\iota*}(\kappa_\iota \tilde\varphi_{\iota^\prime}^*a_{\iota^\prime}))^\mathrm{w}(x, \hbar D)(\varphi_{\iota*}\chi_\iota)\varphi_{\iota*} \nonumber\\
        &=\sum_{U_\iota \cap U_{\iota^\prime}\neq \varnothing} 
        \chi_\iota \varphi_{\iota^\prime}^* (\kappa_\iota a_{\iota^\prime}+O_{S^{m-2}}(\hbar^2))^\mathrm{w}(x, \hbar D)(\varphi_{\iota^\prime *}\chi_\iota)\varphi_{\iota^\prime *} \nonumber\\
        &=\sum_{\iota^\prime \in I} 
        \varphi_{\iota^\prime}^* (a_{\iota^\prime}+O_{S^{m-2}}(\hbar^2))^\mathrm{w}(x, \hbar D)\varphi_{\iota^\prime *}+O_{L^2\to L^2}(\hbar^\infty) \nonumber\\
        &=\sum_{\iota\in I} \Ophloc \left( a_\iota-\hbar b_{1, \iota}\right)+O_{L^2\to L^2}(\hbar^\infty) \label{eq_a0_local_global_II}
    \end{align}
    by changing variables of pseudodifferential operators and the assumption $\rmop{supp}a_\iota \subset \rmop{supp}\varphi_{\iota*}\kappa_\iota$. Here $b_{1, \iota}=b_{1, \iota}(\hbar; x, \xi)\in S^{m-1}(T^*\mathbb{R}^n)$ has an asymptotic expansion 
    \[
        b_{1, \iota}(\hbar; x, \xi)\sim \sum_{j=0}^\infty \hbar^j b_{1j, \iota}(x, \xi), \quad b_{1j, \iota}\in S^{m-j-1}(T^*\mathbb{R}^n)
    \]
    with $\rmop{supp} b_{1j, \iota}\subset \rmop{supp}\varphi_{\iota*}(\kappa_\iota a_0)$. 

    We repeat the same argument for $b_{10, \iota}(x, \xi)$. If we set 
    \[
        a_1(x, \xi):=-\sum_{\iota\in I} \tilde\varphi_\iota^* b_{10, \iota}(x, \xi), 
    \]
    then we have 
    \begin{equation}\label{eq_a1_local_global_II}
        \Oph (a_1)=\sum_{\iota\in I} \Ophloc \left( b_{10, \iota}-\hbar c_{2, \iota}\right)+O_{L^2\to L^2}(\hbar^\infty). 
    \end{equation}
    Here $c_{2, \iota}(\hbar; x, \xi)\in S^{m-2}(T^*\mathbb{R}^n)$ has an asymptotic expansion 
    \[
        c_{2, \iota}(\hbar; x, \xi)\sim \sum_{j=0}^\infty \hbar^j c_{2j, \iota}(x, \xi), \quad c_{2j, \iota}\in S^{m-j-2}(T^*\mathbb{R}^n)
    \]
    with $\rmop{supp} c_{2j, \iota}\subset \rmop{supp}\varphi_{\iota*}(\kappa_\iota a_0)$. 

    Summing up \eqref{eq_a0_local_global_II} and \eqref{eq_a1_local_global_II}$\times \hbar$, we obtain 
    \[
        \Oph (a_0+\hbar a_1)=\sum_{\iota\in I} \Ophloc (a_\iota-\hbar^2 b_{2, \iota})+O_{L^2\to L^2}(\hbar^\infty), 
    \]
    where 
    \[
        b_{2, \iota}(\hbar; x, \xi):=\hbar^{-1}(b_{1, \iota}-b_{10, \iota})+c_{2, \iota} \in S^{m-2}(T^*\mathbb{R}^n). 
    \]
    $b_{2, \iota}(\hbar; x, \xi)$ has an asymptotic expansion 
    \[
        b_{2, \iota}(\hbar; x, \xi)\sim \sum_{j=0}^\infty \hbar^j b_{2j, \iota}(x, \xi), \quad b_{2j, \iota}\in S^{m-j-2}(T^*\mathbb{R}^n)
    \]
    with $\rmop{supp} b_{2j, \iota}\subset \rmop{supp}\varphi_{\iota*}(\kappa_\iota a_0)$. 

    We repeat this argument and construct $a_j\in S^{m-j}_\mathrm{cyl}(T^*M)$ such that 
    \[
        \Oph \left( \sum_{j=0}^N \hbar^j a_j\right)
        =\sum_{\iota\in I} \Ophloc (a_\iota-\hbar^{N+1} b_{N+1, \iota})+O_{L^2\to L^2}(\hbar^\infty)
    \]
    for all $N\in \mathbb{Z}_{\geq 0}$, where $b_{N+1, \iota}(\hbar; x, \xi)\in S^{N+1-j}(T^*\mathbb{R}^n)$ has an asymptotic expansion
    \[
        b_{N+1, \iota}(\hbar; x, \xi)\sim \sum_{j=0}^\infty \hbar^j b_{N+1, j, \iota}(x, \xi), \quad b_{N+1, j, \iota}\in S^{m-j-N-1}(T^*\mathbb{R}^n)
    \]
    with $\rmop{supp} b_{N+1, j, \iota}\subset \rmop{supp}\varphi_{\iota*}(\kappa_\iota a_0)$. 

    The desired symbol $a(\hbar; x, \xi)$ is defined as an asymptotic expansion 
    \[
        a(\hbar; x, \xi)\sim \sum_{j=0}^\infty \hbar^j a_j(x, \xi)
    \]
    by Borel's theorem. 
\end{proof}

\begin{proof}[Proof of Theorem \ref{theo_psido_composition}]
    For $u\in C_c^\infty (M; \Omega^{1/2})$ We decompose $\Oph (a)\Oph (b)$ into 
    \begin{equation}\label{eq_ophab}
        \Oph (a)\Oph (b)
        =\sum_{U_\iota\cap U_{\iota^\prime}\neq \varnothing}
        A_\iota(\varphi_{\iota*}(\chi_\iota \chi_{\iota^\prime}))B_{\iota\iota^\prime}, 
    \end{equation}
    where 
    \begin{equation}\label{eq_defi_aii}
        A_\iota:=\chi_\iota \varphi_\iota^* (\tilde \varphi_{\iota*}(\kappa_\iota a))^\mathrm{w}(x, \hbar D)
    \end{equation}
    and 
    \begin{equation}\label{eq_defi_bii}
        B_{\iota\iota^\prime}u:=\varphi_{\iota*}\varphi_{\iota^\prime}^* (\tilde \varphi_{\iota^\prime *}(\kappa_{\iota^\prime}b))^\mathrm{w}(x, \hbar D)(\varphi_{\iota^\prime *}(\chi_{\iota^\prime}u)). 
    \end{equation}

    Take cylindrical functions $\chi_\iota^\prime\in C^\infty (M)$ such that $\rmop{supp}\chi_\iota^\prime\subset U_\iota$ and $\chi_\iota^\prime=1$ near $\rmop{supp}\chi_\iota$. 

    We treat $A_\iota$. In local coordinates,  
    \begin{align*}
        (\varphi_{\iota*}\kappa_\iota)(\tilde \varphi_{\iota*}(\chi_\iota^\prime a))^\mathrm{w}
        &=((\varphi_{\iota*}\kappa_\iota)\# (\tilde\varphi_{\iota*}(\chi_\iota^\prime a)))^\mathrm{w} \\
        &=\left( \tilde\varphi_{\iota*}\left(\kappa_\iota a+\frac{i\hbar}{2}\{ \kappa_\iota, a\}\right)+O_{S^{m_1-2}}(\hbar^2)\right)^\mathrm{w}. 
    \end{align*}
    Thus 
    \begin{equation}\label{eq_aii_decomposition}
        A_\iota
        =\underbrace{\kappa_\iota \varphi_\iota^* (\tilde\varphi_{\iota*}(\chi_\iota^\prime a))^\mathrm{w}}_{=:A_\iota^\prime} 
        -\frac{i\hbar}{2}\underbrace{\chi_\iota \varphi_\iota^* (\tilde\varphi_{\iota*}\{ \kappa_\iota, a\}+\hbar c_\iota^\prime)^\mathrm{w}}_{=:A_\iota^\pprime}. 
    \end{equation}
    Here $c_\iota^\prime (\hbar; x, \xi)\in S^{m_1-2}(T^*\mathbb{R}^n)$ satisfies $\rmop{supp}c_\iota^\prime \subset \rmop{supp}\rmop{supp} \tilde\varphi_{\iota*}(\kappa_\iota a)$ modulo $O(\hbar^\infty)$. 

    We calculate $A_\iota^\prime (\varphi_{\iota*}(\chi_\iota \chi_{\iota^\prime}))B_{\iota\iota^\prime}$ and $A_\iota^\pprime (\varphi_{\iota*}(\chi_\iota \chi_{\iota^\prime}))B_{\iota\iota^\prime}$ respectively. 

    \fstep{Calculation of $\bm{A_\iota^\prime (\varphi_{\iota*}(\chi_\iota \chi_{\iota^\prime}))B_{\iota\iota^\prime}}$}We apply the changing variables for Weyl quantization acting on half densities to $B_{\iota\iota^\prime}$ and obtain 
    \begin{equation}\label{eq_aipb_wip}
        \begin{split}
            (\varphi_{\iota*}(\chi_\iota \chi_{\iota^\prime}))B_{\iota\iota^\prime}
            &=(\varphi_{\iota*}(\chi_\iota \chi_{\iota^\prime}))\varphi_{\iota*}\varphi_{\iota^\prime}^* (\tilde \varphi_{\iota^\prime *}(\chi_\iota^\prime\kappa_{\iota^\prime}b))^\mathrm{w}(x, \hbar D)(\varphi_{\iota^\prime *}(\chi_\iota^\prime\chi_{\iota^\prime}))\varphi_{\iota^\prime *} \\
            &\quad+O_{L^2\to L^2}(\hbar^\infty) \\
            &=(\varphi_{\iota*}(\chi_\iota \chi_{\iota^\prime}))b_{\iota\iota^\prime}^\mathrm{w}(x, \hbar D)(\varphi_{\iota*}(\chi_\iota^\prime \chi_{\iota^\prime}))\varphi_{\iota*}+O_{L^2\to L^2}(\hbar^\infty), 
        \end{split}
    \end{equation}
    where 
    \[
        b_{\iota\iota^\prime}(x, \xi)
        =(\tilde\varphi_{\iota^\prime}\circ \tilde\varphi_\iota^{-1})^*(\tilde\varphi_{\iota^\prime *}( \chi_\iota^\prime\kappa_{\iota^\prime}b))+\hbar^2 q^\prime_{\iota\iota^\prime} 
        =\tilde\varphi_{\iota *}( \chi_\iota^\prime \kappa_{\iota^\prime}b)+\hbar^2 q^\prime_{\iota\iota^\prime}
    \]
    and $q^\prime_{\iota\iota^\prime}(\hbar; x, \xi)\in S^{m_2-2}(T^*\mathbb{R}^n)$ satisfies $\rmop{supp} q^\prime_{ \iota\iota^\prime}\subset \rmop{supp}\tilde\varphi_{\iota*}(\chi_\iota \kappa_{\iota^\prime}b)$ modulo $O(\hbar^\infty)$. 

    Hence by \eqref{eq_aii_decomposition}, we have 
\begin{equation}\label{eq_aii1_bii}
    \begin{split}
    &A_\iota^\prime (\varphi_{\iota*}(\chi_\iota \chi_{\iota^\prime}))B_{\iota\iota^\prime} \\
    &=\kappa_\iota \varphi_\iota^* (\tilde\varphi_{\iota*}(\chi_\iota^\prime a))^\mathrm{w}
    (\varphi_{\iota*}(\chi_\iota \chi_{\iota^\prime}))
    (\tilde\varphi_{\iota *}( \chi_\iota^\prime \kappa_{\iota^\prime}b)+\hbar^2 q^\prime_{\iota\iota^\prime})^\mathrm{w}\varphi_{\iota*}(\chi_\iota^\prime \chi_{\iota^\prime})\varphi_{\iota*} \\
    &\quad+O_{L^2\to L^2}(\hbar^\infty) \\
    &=\kappa_\iota \varphi_\iota^* ((\tilde\varphi_{\iota*}(\chi_\iota^\prime a))\# (\varphi_{\iota*}(\chi_\iota \chi_{\iota^\prime}))\# (\tilde\varphi_{\iota *}( \chi_\iota^\prime \kappa_{\iota^\prime}b)+\hbar^2 q^\prime_{\iota\iota^\prime}))^\mathrm{w}\varphi_{\iota*}(\chi_\iota^\prime \chi_{\iota^\prime})\varphi_{\iota*} \\
    &\quad+O_{L^2\to L^2}(\hbar^\infty) \\
    &=\kappa_\iota \varphi_\iota^* \biggl(\tilde\varphi_{\iota*}\biggl(\chi_\iota \kappa_{\iota^\prime}ab
    +\frac{i\hbar}{2}(\{a, b\} \chi_\iota \kappa_{\iota^\prime}+\{a, \chi_\iota \kappa_{\iota^\prime}\} b+\{\chi_\iota, b\}\kappa_{\iota^\prime}a)\biggr)
    +\hbar^2 \tilde q^\prime_{\iota\iota^\prime}\biggr)^\mathrm{w}
    \\
    &\quad (\varphi_{\iota*}(\chi_\iota^\prime \chi_{\iota^\prime}))\varphi_{\iota*}+O_{L^2\to L^2}(\hbar^\infty) \\
    &=\kappa_\iota \varphi_\iota^* \biggl(\tilde\varphi_{\iota*}\biggl(\chi_\iota \kappa_{\iota^\prime}ab
    +\frac{i\hbar}{2}(\{a, b\} \chi_\iota \kappa_{\iota^\prime}+\{a,  \kappa_{\iota^\prime}\} \chi_\iota b)\biggr)
    +\hbar^2 \tilde q^\prime_{\iota\iota^\prime}\biggr)^\mathrm{w}
    (\varphi_{\iota*}\chi_\iota^\prime )\varphi_{\iota*}
    \\
    &\quad +O_{L^2\to L^2}(\hbar^\infty). 
    \end{split}
\end{equation}
Here $\tilde q^\prime_{\iota\iota^\prime}(\hbar; x, \xi)\in S^{m_2-2}(T^*\mathbb{R}^n)$ satisfies $\rmop{supp} \tilde q^\prime_{ \iota\iota^\prime}\subset \rmop{supp}\tilde\varphi_{\iota*}(\chi_\iota \kappa_{\iota^\prime}ab)$ modulo $O(\hbar^\infty)$. For fixed $\iota\in I$, we sum \eqref{eq_aii1_bii} over $\iota^\prime \in I$ such that $U_\iota\cap U_{\iota^\prime}\neq \varnothing$ and obtain 
    \begin{equation}\label{eq_aii1_bii_result}
        \begin{split}
        &\sum_{\iota^\prime: U_\iota \cap U_{\iota^\prime}\neq \varnothing}A_\iota^\prime (\varphi_{\iota*}(\chi_\iota \chi_{\iota^\prime}))B_{\iota\iota^\prime} \\
        &=\kappa_\iota \varphi_\iota^* \biggl(\tilde\varphi_{\iota*}\biggl(\chi_\iota ab+\frac{i\hbar}{2}\{a, b\} \chi_\iota\biggr)+\hbar^2 \tilde q^\prime_\iota\biggr)^\mathrm{w}(\varphi_{\iota*}\chi_\iota^\prime)\varphi_{\iota*} \\
        &\quad+O_{L^2\to L^2}(\hbar^\infty). 
        \end{split}
    \end{equation}
    Here $\tilde q^\prime_\iota:=\sum_{\iota^\prime: U_\iota \cap U_{\iota^\prime}\neq \varnothing}\tilde q^\prime_{\iota\iota^\prime}$. 

    Since $\tilde q^\prime_\iota=\chi_\iota^2 \tilde q^\prime_\iota+O_{S^0}(\hbar^\infty)$, we can find a symbol $\tilde c^\prime_\iota (\hbar; x, \xi)\in S^{m_1+m_2-2}(T^*\mathbb{R}^n)$ which satisfies
\begin{align*}
    &(\varphi_{\iota*}\kappa_\iota)\left(\tilde\varphi_{\iota*}\left(\chi_\iota ab+\frac{i\hbar}{2}\{a, b\} \chi_\iota\right)+\hbar^2 \tilde q^\prime_\iota \right)^\mathrm{w} \\
    &=(\varphi_{\iota*}\chi_\iota)\left(\tilde\varphi_{\iota*}\left(\kappa_\iota ab+\frac{i\hbar}{2}\{a, b\} \kappa_\iota+\frac{i\hbar}{2}\{\kappa_\iota, ab\} \right)+\hbar^2 c^\prime_\iota\right)^\mathrm{w}(\varphi_{\iota*}\chi_\iota) \\
    &\quad+O_{L^2\to L^2}(\hbar^\infty)
\end{align*}
and $\rmop{supp}c^\prime_{j\iota}\subset \rmop{supp}\tilde\varphi_{\iota*}(\kappa_\iota ab)$ modulo $O(\hbar^\infty)$. Hence \eqref{eq_aii1_bii_result} becomes 
\begin{align*}
    &\sum_{\iota^\prime: U_\iota \cap U_{\iota^\prime}\neq \varnothing}A_\iota^\prime (\varphi_{\iota*}(\chi_\iota \chi_{\iota^\prime}))B_{\iota\iota^\prime} \\
    &=\chi_\iota\varphi_\iota^*\left(\tilde\varphi_{\iota*}\left(\kappa_\iota ab+\frac{i\hbar}{2}\{a, b\} \kappa_\iota+\frac{i\hbar}{2}\{\kappa_\iota, ab\} \right) +\hbar^2 c^\prime_\iota \right)^\mathrm{w}(\varphi_{\iota*}\chi_\iota)\varphi_{\iota*} \\
    &\quad+O_{L^2\to L^2}(\hbar^\infty).  
\end{align*}
Summing up this over $\iota\in I$ and obtain 
\begin{equation}
    \label{eq_aii1_bii_final}
    \begin{split}
        &\sum_{U_\iota \cap U_{\iota^\prime}\neq \varnothing}A_\iota^\prime (\varphi_{\iota*}(\chi_\iota \chi_{\iota^\prime}))B_{\iota\iota^\prime} \\
    &=\Oph \left( ab+\frac{i\hbar}{2}\{a, b\} \right)
    +\sum_{\iota\in I} \Ophloc \left( \frac{i\hbar}{2}\tilde\varphi_{\iota*}\{\kappa_\iota, ab\}+\hbar^2 c^\prime_\iota\right) \\
    &\quad+O_{L^2\to L^2}(\hbar^\infty). 
    \end{split}
\end{equation}
Since the support of $\tilde\varphi_{\iota*}\{\kappa_\iota, ab\}$ and $c^\prime_\iota$ is included in $\rmop{supp}\tilde\varphi_{\iota*}(\kappa_\iota ab)$, we can apply Lemma \ref{lemm_psido_local_global} for \eqref{eq_aii1_bii_final} and find a symbol $c^\prime(\hbar; x, \xi)\in S^{m_1+m_2-2}_\mathrm{cyl}(T^*M)$ which satisfies 
\[
    \sum_{\iota\in I} \Ophloc \left( \frac{i\hbar}{2}\tilde\varphi_{\iota*}\{\kappa_\iota, ab\}+\hbar^2 c^\prime_\iota\right)=\hbar \Oph (c^\prime)+O_{L^2\to L^2}(\hbar^\infty), 
\]
$\rmop{supp}c^\prime_j\subset \rmop{supp}(ab)$ modulo $O(\hbar^\infty)$ and 
\[
    c^\prime_0(x, \xi)=\sum_{\iota\in I} \frac{i}{2}\{\kappa_\iota , ab\}=0. 
\]

Thus \eqref{eq_aii1_bii_final} becomes
\begin{equation}
    \label{eq_aii1_bii_final_2}
    \begin{split}
    &\sum_{U_\iota \cap U_{\iota^\prime}\neq \varnothing}A_\iota^\prime (\varphi_{\iota*}(\chi_\iota \chi_{\iota^\prime}))B_{\iota\iota^\prime} \\
    &=\Oph \left( ab+\frac{i\hbar}{2}\{a, b\} +\hbar^2 (\hbar^{-1}c^\prime)\right)
    +O_{L^2\to L^2}(\hbar^\infty). 
    \end{split}
\end{equation}

\fstep{Calculation of $\bm{A_\iota^\pprime (\varphi_{\iota*}(\chi_\iota \chi_{\iota^\prime}))B_{\iota\iota^\prime}}$}
It is enough to calculate the principal term of $A_{\iota\iota^\prime}^\pprime B_{\iota\iota^\prime}$ in \eqref{eq_aii_decomposition} since $A_{\iota\iota^\prime}^\pprime$ has a coefficient $\hbar$. By changing variables of the Weyl quantization acting on half-densities, we have 
\[
    A_\iota^\pprime=\varphi_{\iota^\prime}^* \left( \tilde\varphi_{\iota^\prime *} (\chi_{\iota^\prime}^\prime\{ \kappa_\iota, a\})+O_{S^{m_1-2}}(\hbar)\right)^\mathrm{w}(\varphi_{\iota^\prime}\circ \varphi_\iota^{-1})_*+O_{L^2\to L^2}(\hbar^\infty). 
\]
Hence
\begin{align*}
    &A_{\iota\iota^\prime}^\pprime (\varphi_{\iota*}(\chi_\iota \chi_{\iota^\prime}))B_{\iota\iota^\prime} \\
    &=\varphi_{\iota^\prime}^* \left( \tilde\varphi_{\iota^\prime *} (\chi_{\iota^\prime}^\prime\{ \kappa_\iota, a\})+O_{S^{m_1-2}}(\hbar)\right)^\mathrm{w}
    (\varphi_{\iota*}(\chi_\iota \chi_{\iota^\prime}))
    \left(\tilde \varphi_{\iota^\prime *}(\kappa_{\iota^\prime}b)\right)^\mathrm{w} (\varphi_{\iota^\prime *}\chi_{\iota^\prime})\varphi_{\iota^\prime*} \\
    &\quad +O_{L^2\to L^2}(\hbar^\infty) \\
    &=\chi_{\iota^\prime}\varphi_{\iota^\prime}^*(\tilde\varphi_{\iota^\prime*} (\kappa_{\iota^\prime}b\{ \kappa_\iota, a\})+\hbar c^\pprime_{\iota\iota^\prime})^\mathrm{w}\varphi_{\iota^\prime}(\chi_{\iota^\prime}u)+O_{L^2\to L^2}(\hbar^\infty). 
\end{align*}
$c^\pprime_{\iota\iota^\prime}(\hbar; x, \xi)\in S^{m_1+m_2-2-j}(T^*\mathbb{R}^n)$ satisfies $\rmop{supp} c^\pprime_{\iota\iota^\prime}\subset \rmop{supp}\tilde\varphi_{\iota*}(\kappa_\iota \kappa_{\iota^\prime} ab)$ modulo $O(\hbar^\infty)$. We sum them up over $\iota\in I$ such that $U_\iota\cap U_{\iota^\prime}\neq \varnothing$. Then the terms including $\{\kappa_\iota, a\}$ vanish and we obtain 
    \begin{equation}\label{eq_aii2_bii_final}
    \sum_{\iota: U_\iota\cap U_{\iota^\prime}\neq \varnothing}A_{\iota\iota^\prime}^\pprime (\varphi_{\iota*}(\chi_\iota \chi_{\iota^\prime}))B_{\iota\iota^\prime} 
    =\hbar\chi_{\iota^\prime}\varphi_{\iota^\prime}^*(c^\pprime_{\iota^\prime})^\mathrm{w}(\varphi_{\iota^\prime*}\chi_{\iota^\prime})\varphi_{\iota^\prime*}+O_{L^2\to L^2}(\hbar^\infty), 
\end{equation}
where $c^\pprime_{\iota^\prime}:=\sum_{\iota: U_\iota\cap U_{\iota^\prime}\neq \varnothing}c^\pprime_{\iota\iota^\prime}$. The sum of \eqref{eq_aii2_bii_final} over $\iota^\prime \in I$ is 
\begin{equation}\label{eq_aii2_bii_final_2}
    \sum_{U_\iota\cap U_{\iota^\prime}\neq \varnothing}A_{\iota\iota^\prime}^\pprime (\varphi_{\iota*}(\chi_\iota \chi_{\iota^\prime}))B_{\iota\iota^\prime} 
    =\hbar\sum_{\iota^\prime \in I}\Ophloc[\iota^\prime] (c^\pprime_{\iota^\prime})+O_{L^2\to L^2}(\hbar^\infty). 
\end{equation}
Since $\rmop{supp}c^\pprime_{\iota^\prime}\subset \rmop{supp} \tilde\varphi_{\iota^\prime *}(\kappa_{\iota^\prime}ab)$, we can apply Lemma \ref{lemm_psido_local_global} and find a symbol $c^\pprime (\hbar; x, \xi)\in S^{m_1+m_2-2}_\mathrm{cyl}(T^*M)$ which satisfies 
\[
    \sum_{\iota^\prime \in I}\Ophloc[\iota^\prime] (c^\pprime_{\iota^\prime})
    =\Oph (c^\pprime)+O_{L^2\to L^2}(\hbar^\infty)
\]
and $\rmop{supp}c^\pprime_j \subset \rmop{supp} (ab)$ modulo $O(\hbar^\infty)$. Hence \eqref{eq_aii2_bii_final_2} becomes 
\begin{equation}
    \label{eq_aii2_bii_final_3}
    \sum_{U_\iota\cap U_{\iota^\prime}\neq \varnothing}A_{\iota\iota^\prime}^\pprime (\varphi_{\iota*}(\chi_\iota \chi_{\iota^\prime}))B_{\iota\iota^\prime} 
    =\hbar \Oph (c^\pprime)+O_{L^2\to L^2}(\hbar^\infty). 
\end{equation}

\fstep{Conclusion}\eqref{eq_aii_decomposition}, \eqref{eq_aii1_bii_final_2} and \eqref{eq_aii2_bii_final_3} imply 
\begin{equation}\label{eq_psido_composition_wip}
    \begin{split}
        &\Oph (a)\Oph (b) \\
        &=\Oph \left(ab+\frac{i\hbar}{2}\{a, b\}+\hbar^2 (\hbar^{-1}c^\prime)-\frac{i\hbar}{2}(\hbar c^\pprime)\right)+O_{L^2\to L^2}(\hbar^\infty) \\
        &=\Oph \left(ab+\frac{i\hbar}{2}\{a, b\}+\hbar^2 c\right)+O_{L^2\to L^2}(\hbar^\infty)
    \end{split}
\end{equation}
where 
\[
    c(\hbar; x, \xi):=\hbar^{-1}c^\prime(\hbar; x, \xi)-\frac{i}{2}c^\pprime(\hbar; x, \xi). 
\]
The symbol $c\in S^{m_1+m_2-2}_\mathrm{cyl}(T^*M)$ has the desired properties. 
\end{proof}

Next we prove the sharp G\aa rding inequality (Theorem \ref{theo_sharp_garding}). We begin with the case of Euclidean spaces. 

\begin{theo}[Sharp G\aa rding inequality on Euclidean spaces]\label{theo_sharp_garding_euclid}
    For all $a\in S^0(\mathbb{R}^n)$ with $\rmop{Re}a\geq 0$, there exists a symbol $b=b(\hbar)\in S^0(\mathbb{R}^n)$ such that the following statements hold: 
    \begin{itemize}
        \item The inequality 
        \begin{equation}\label{eq_sharp_garding_euclid}
            \rmop{Re}a^\mathrm{w}(x, \hbar D)\geq -\hbar \rmop{Re}b^\mathrm{w}(x, \hbar D)
        \end{equation} 
        holds. 
        \item $\rmop{supp}b\subset \rmop{supp}a$ modulo $O(\hbar^\infty)$. 
    \end{itemize}
    Furthermore, if the symbol $a$ also depends on some parameter $\tau\in \Omega$ and are uniformly bounded in $S^0(T^*\mathbb{R}^n)$, then the symbol $b\in S^0(T^*\mathbb{R}^n)$ itself are uniformly bounded with respect to $\tau\in\Omega$. 
\end{theo}

For investigation of the support of $b(\hbar; x, \xi)$ in \eqref{theo_sharp_garding_euclid}, we recall the FBI transform and its fundamental properties. 

\begin{prop}\label{prop_fbi_fundamental}
    We define an FBI transform $Fu$ of $u\in \mathscr{S}(\mathbb{R}^n)$ as 
\[
F u(x, \xi):=\frac{2^{n/4}}{(2\pi \hbar)^{3n/4}}\int_{\mathbb{R}^n}e^{-|x-y|^2/2\hbar+i\xi\cdot(x-y)/\hbar}u(y)\, \diff y. 
\]
Then the following statements hold. 
\begin{enumerate}
    \renewcommand{\labelenumi}{(\roman{enumi})}
    \item $F$ is continuously extended to a linear isometry from $L^2(\mathbb{R}^n)$ to $L^2(\mathbb{R}^{2n})$. 
    \item For $b\in S^0(T^*\mathbb{R}^n)$, we define
    \[
    p_b(x, \xi):=\left(\frac{1}{\pi\hbar}\right)^n\int_{\mathbb{R}^{2n}} e^{-|x-y|^2/\hbar - |\xi-\eta|^2/\hbar}b(y, \eta)\, \diff y\diff \eta. 
    \]
    Then $p_b\in S^0(T^*\mathbb{R}^n)$ and 
    \begin{equation}\label{eq_weyl_antiwick}
    F^* M_b F =p_b^\mathrm{w}(x, \hbar D). 
    \end{equation}
    Here $M_b: u\mapsto b u$ is the multiplication operator by $b$. 
\end{enumerate}
\end{prop}

\begin{rema*}
    $F^*M_bF$ is so-called anti-Wick quantization of the symbol $b$. 
\end{rema*}

\begin{proof}
    A direct calculation shows (i) and the relation \eqref{eq_weyl_antiwick} (see \cite{Martinez02} or Chapter 13 in \cite{Zworski12} for details). We have to prove $p_b\in S^0(T^*\mathbb{R}^n)$ if $b\in S^0(T^*\mathbb{R}^n)$. The facts $\partial_x e^{-|x-y|^2/\hbar}=-\partial_y e^{-|x-y|^2/\hbar}$, $\partial_\xi e^{-|\xi-\eta|^2/\hbar}=-\partial_\eta e^{-|\xi-\eta|^2/\hbar}$ and integration by parts show that 
    \[
        \partial_x^\alpha\partial_\xi^\beta p_b(x, \xi)=p_{\partial_x^\alpha\partial_\xi^\beta b}(x, \xi). 
    \]
    Thus the estimate 
    \begin{align*}
        |\partial_x^\alpha\partial_\xi^\beta p_b(x, \xi)|
        &\leq \frac{|b|_{0, \alpha, \beta}}{(\pi\hbar)^n} \int_{\mathbb{R}^{2n}} e^{-|x-y|^2/\hbar - |\xi-\eta|^2/\hbar}\jbracket{\eta}^{-|\beta|}\, \diff y\diff \eta \\
        &\leq C|b|_{0, \alpha, \beta}\jbracket{\xi}^{-|\beta|}, 
    \end{align*}
    where 
    \[
        |b|_{0, \alpha, \beta}:=\sup_{(x, \xi)\in T^*\mathbb{R}^n} \jbracket{\xi}^{|\beta|}|\partial_x^\alpha \partial_\xi^\beta b(x, \xi)|.   
    \]
    This shows that $p_b\in S^0(T^*\mathbb{R}^n)$. 
\end{proof}

\begin{proof}[Proof of Theorem \ref{theo_sharp_garding_euclid}]
    We define a symbol $b(\hbar; x, \xi)\in S^0(T^*\mathbb{R}^n)$ as 
    \[
    b(\hbar; x, \xi)=\hbar^{-1}(p_a(\hbar; x, \xi)-a(x, \xi)). 
    \]
    Then by $\rmop{Re}a\geq 0$, we have
    \begin{align*}
        &\rmop{Re}\jbracket{a^\mathrm{w}(x, \hbar D)u, u}_{L^2(\mathbb{R}^n)} \\
        &=\jbracket{M_{\rmop{Re}a} F u, F u}_{L^2(\mathbb{C}^n)}-\hbar \rmop{Re}\jbracket{b^\mathrm{w}(\hbar; x, \hbar D) u, u}_{L^2(\mathbb{R}^n)} \\
    &\geq -\hbar \rmop{Re}\jbracket{b^\mathrm{w}(\hbar; x, \hbar D) u, u}_{L^2(\mathbb{R}^n)}. 
\end{align*}

    By a calculation by the Taylor theorem, we obtain
    \begin{align*}
        p_a(x, \xi)&=\sum_{j=0}^N\frac{1}{j!}\left(\frac{\hbar}{4}\right)^j \triangle_{x, \xi}^j a(x, \xi)+\hbar^{N+1}q_{N+1}(\hbar; x, \xi), \\
    q_{N+1}(\hbar; x, \xi)&:=\sum_{|\alpha|+|\beta|=2N+2}\pi^{-n}\int_{\mathbb{R}^{2n}} \diff y\diff \eta \, e^{-|y|^2-|\eta|^2} y^\alpha \eta^\beta  \\
    &\quad\times \int_0^1 \diff \tau \, \partial_x^\alpha \partial_\xi^\beta a(x+\hbar^{1/2}\tau y, \xi+\hbar^{1/2}\tau \eta)\end{align*}
    Thus
    \[
         b(\hbar; x, \xi)=\sum_{j=0}^{N-1}\frac{1}{(j+1)!}\frac{\hbar^j}{4^{j+1}} \triangle_{x, \xi}^{j+1} a(x, \xi)+\hbar^N q_{N+1}(\hbar; x, \xi). \qedhere
    \]
    This implies $\rmop{supp}b\subset \rmop{supp}a$ modulo $O(\hbar^\infty)$. 
\end{proof}

\begin{proof}[Proof of Theorem \ref{theo_sharp_garding}]
    Let $u\in C_c^\infty (M; \Omega^{1/2})$. Since
    \begin{equation}\label{eq_expectation_decomposition}
        \rmop{Re}\jbracket{ \Oph (a)u, u}_{L^2}=\sum_{\iota\in I}\rmop{Re}\jbracket{ (\tilde\varphi_{\iota*}(\kappa_\iota a))^\mathrm{w}(x, \hbar D)(\varphi_{\iota*}(\chi_\iota u)), \varphi_{\iota*}(\chi_\iota u)}_{L^2}, 
    \end{equation}
    it is enough to investigate $(\tilde\varphi_{\iota*}(\kappa_\iota a))^\mathrm{w}(x, \hbar D)$ for each $\iota\in I$. 
    By Theorem \ref{theo_sharp_garding_euclid}, there exists $b_\iota=b_\iota(\hbar)\in S^0(T^*\mathbb{R}^n)$ such that 
    \begin{align*}
        &\rmop{Re}\jbracket{ (\tilde\varphi_{\iota*}(\kappa_\iota a))^\mathrm{w}(x, \hbar D)\varphi_{\iota*}(\chi_\iota u), \varphi_{\iota*}(\chi_\iota u)}_{L^2} \\
        &\geq -\hbar \jbracket{b_\iota^\mathrm{w}(x, \hbar D)\varphi_{\iota*}(\chi_\iota u), \varphi_{\iota*}(\chi_\iota u)}_{L^2}
    \end{align*}
        and $\rmop{supp}b_\iota \subset (\tilde\varphi_{\iota*}(\kappa_\iota a))$ modulo $O(\hbar^\infty)$. 
    \[
        b_\iota (\hbar; x, \xi)\sim \sum_{j=0}^\infty\frac{1}{(j+1)!}\frac{\hbar^j}{4^{j+1}} \triangle_{x, \xi}^{j+1} (\tilde\varphi_{\iota*}(\kappa_\iota a))(x, \xi) \quad \text{in } S^0(T^*\mathbb{R}^n). 
    \]
        Thus by \eqref{eq_expectation_decomposition}, we obtain 
        \begin{align*}
            \rmop{Re}\jbracket{ \Oph (a)u, u}_{L^2}
            &\geq -\hbar \sum_{\iota\in I}\rmop{Re}\jbracket{ b_\iota^\mathrm{w}(x, \hbar D)(\varphi_{\iota*}(\chi_\iota u)), \varphi_{\iota*}(\chi_\iota u)}_{L^2} \\
            &=-\hbar \jbracket{\sum_{\iota\in I} \Ophloc (b_\iota) u, u}_{L^2}. 
        \end{align*}
        Since $\rmop{supp}b_\iota \subset \rmop{supp}(\tilde\varphi_{\iota*}(\kappa_\iota a))$ modulo $O(\hbar^\infty)$, we can apply Lemma \ref{lemm_psido_local_global} and obtain a symbol $b(\hbar; x, \xi)\in S^0_\mathrm{cyl}(T^*M)$ which satisfies
        \[
            \sum_{\iota\in I} \Ophloc (b_\iota)=\Oph (b)+O_{L^2\to L^2}(\hbar^\infty)
        \]
        and $\rmop{supp}b\subset \rmop{supp}a$ modulo $O(\hbar^\infty)$. 
\end{proof}

\subsection{Non-canonical quantization and (radially homogeneous) wavefront sets}\label{subs_quantization_wf_hwf}

In this section we prove Theorem \ref{theo_hwf_quantization} and Proposition \ref{prop_wf_quantization}. As a preparation, we prove a lemma on the relation between a quantization of locally defined symbols and the quantization procedure $\Oph$. 

\begin{lemm}\label{lemm_local_quantization}
    Let $\varphi: U\to V$ be polar coordinates on $M$ and $\chi\in C^\infty (M)$ be a cylindrical function supported in $U$. Then, for a symbol $b\in S^m(T^*\mathbb{R}^n)$, there exists $a(\hbar; x, \xi)\in S^m_\mathrm{cyl}(T^*M)$ which satisfies 
        \[
            \chi\varphi^* b^\mathrm{w}(x, \hbar D)(\varphi_* \chi)\varphi_*
            =\Oph (a)+O_{L^2\to L^2}(\hbar^\infty)
        \]
        and has an asymptotic expansion 
        \[
            a(\hbar; x, \xi)\sim \sum_{j=0}^\infty \hbar^j a_j(x, \xi), \quad a_j\in S^{m-j}_\mathrm{cyl}(T^*M)
        \]
        with $\rmop{supp}a_j \subset \rmop{supp}(\chi^2\tilde\varphi^* b)$ and $a_0(x, \xi)=\chi(x)^2\tilde\varphi^*b(x, \xi)$. 
\end{lemm}

\begin{proof}
    We calculate the composition 
    \begin{align*}
        \chi\varphi^* b^\mathrm{w}(x, \hbar D)(\varphi_* \chi)\varphi_*
        = \varphi^* b_\chi^\mathrm{w}(x, \hbar D)\varphi_*. 
    \end{align*}
    Here $b_\chi (\hbar; x, \xi)\in S^m(T^*\mathbb{R}^n)$ has an asymptotic expansion
    \[
        b_\chi (\hbar; x, \xi)\sim \sum_{j=0}^\infty \hbar^j b_{\chi, j}(x, \xi), \quad b_{\chi, j}\in S^{m-j}(T^*\mathbb{R}^n)
    \]
    with $\rmop{supp} b_{\chi, j}\subset (b(\varphi_*\chi))$ and $b_0(x, \xi)=\chi(x)^2b(x, \xi)$. 
    We decompose $\varphi^* b_\chi^\mathrm{w}\varphi_*$ into 
    \begin{equation}\label{eq_psido_local_decomposition}
        \varphi^* b_\chi^\mathrm{w}(\hbar; x, \hbar D)\varphi_*=\sum_{\iota: U_\iota \cap U\neq \varnothing} \varphi^* ((\tilde\varphi_*\kappa_\iota)b_\chi)^\mathrm{w}(\hbar; x, \hbar D)\varphi_*. 
    \end{equation}
    By the changing variables of pseudodifferential operators and $b_{\chi, 0}=\chi^2 \tilde\varphi^* b$, we have
    \begin{equation}
        \label{eq_psido_local_changing}
        ((\tilde\varphi_*\kappa_\iota)b_\chi)^\mathrm{w}(\hbar; x, \hbar D)
        =\varphi_*\varphi_\iota^* (\tilde \varphi_{\iota*}(\kappa_\iota\chi^2 \tilde\varphi^*b)+\hbar c_\iota )^\mathrm{w}(\hbar; x, \hbar D) \varphi_{\iota*}\varphi^*, 
    \end{equation}
    where $c_\iota (\hbar; x, \xi)\in S^{m-1}(T^*\mathbb{R}^n)$ and has an asymptotic expansion
    \[
            c_\iota (\hbar; x, \xi)\sim \sum_{j=0}^\infty \hbar^j c_{j, \iota}(x, \xi), \quad c_{j, \iota}\in S^{m-1-j}(T^*\mathbb{R}^n)
        \]
        with $\rmop{supp}c_{j, \iota} \subset \rmop{supp}\tilde\varphi_{\iota*} (\kappa_\iota \chi^2\tilde\varphi^*b)$. 
    Substituting \eqref{eq_psido_local_changing} to \eqref{eq_psido_local_decomposition}, we obtain 
    \begin{equation}\label{eq_psido_local_iota}
    \begin{split}
        &\varphi^* b_\chi^\mathrm{w}(\hbar; x, \hbar D)\varphi_* \\
        &=\sum_{\iota: U_\iota \cap U\neq \varnothing} \varphi_\iota^* (\tilde \varphi_{\iota*}(\kappa_\iota \varphi^*b_\chi)+\hbar c_\iota )^\mathrm{w}(\hbar; x, \hbar D) \varphi_{\iota*}\\
        &=\sum_{\iota\in I} \Ophloc (\tilde \varphi_{\iota*}(\kappa_\iota \varphi^*b_\chi)+\hbar c_\iota )+O_{L^2\to L^2}(\hbar^\infty). \\
    \end{split}
\end{equation}
    Since the support of $\tilde \varphi_{\iota*}(\kappa_\iota \varphi^*b_\chi)+\hbar c_\iota$ is included in $\rmop{supp}\varphi_{\iota*}\kappa_\iota$, we can apply Lemma \ref{lemm_psido_local_global} for \eqref{eq_psido_local_iota} and obtain a symbol $a(\hbar; x, \xi)\in S^m_\mathrm{cyl}(T^*M)$ which satisfies 
    \[
        \sum_{\iota\in I} \Ophloc (\tilde \varphi_{\iota*}(\kappa_\iota \varphi^*b)+\hbar c_\iota )
        =\Oph (a)+O_{L^2\to L^2}(\hbar^\infty)
    \]
    and has an asymptotic expansion 
    \[
        a(\hbar; x, \xi)\sim \sum_{j=0}^\infty \hbar^j a_j(x, \xi), \quad a_j\in S^{m-j}_\mathrm{cyl}(T^*M)
    \]
    with $\rmop{supp}a_j \subset \rmop{supp}(\chi^2\tilde\varphi^* b)$ and $a_0(x, \xi)=\chi(x)^2\varphi^*b(x, \xi)$. 
\end{proof}

\begin{proof}[Proof of Theorem \ref{theo_hwf_quantization}]
    Take a cylindrical function $\chi^\prime\in C^\infty (M)$ such that $\rmop{supp}\chi^\prime \subset U$ and $\chi^\prime=1$ near $\rmop{supp}\chi$. By Lemma \ref{lemm_local_quantization}, there exists a symbol $c(\hbar; x, \xi)\in S^0_\mathrm{cyl}(T^*M)$ which satisfies
\begin{equation}\label{eq_hwf_time_difference}
    \chi\varphi^*(\varphi_*\chi^\prime-a)^\mathrm{w}(\hbar r, \theta, \hbar D_r, \hbar D_\theta)\varphi_*(\chi u)=\Oph (c)u
\end{equation}
and has an asymptotic expansion 
\[
    c(\hbar; x, \xi)\sim \sum_{j=0}^\infty \hbar^j c_j(\hbar; x, \xi), \quad c_j(\hbar; x, \xi)\in S^{-j}_\mathrm{cyl}(T^*M)
\]
with $\rmop{supp}c_j(\hbar) \subset \rmop{supp}\chi^2(1-\tilde\varphi_\iota^*a (\hbar r, \theta, \rho, \eta))$. We compose $A_\hbar (t_0)$ to the left hand side of \eqref{eq_hwf_time_difference} and obtain 
\begin{equation}\label{eq_hwf_time_difference_2}
    A_\hbar (t_0)(\chi^2 u)-A_\hbar (t_0)\chi\varphi^*a^\mathrm{w}(\hbar r, \theta, \hbar D_r, \hbar D_\theta)\varphi_*(\chi u)=A_\hbar (t_0)\Oph (c)u
\end{equation}
    Taking $\delta_0, \delta_1, \ldots$ sufficiently small, we can assume that 
    \[
        \rmop{supp} a(\hbar^{-1}t_0)\cap \rmop{supp}(\chi^\prime-\tilde\varphi^*a(\hbar r, \theta, \rho, \eta))=\varnothing
    \]
    and 
    \[
        \rmop{supp}a(\hbar^{-1}t_0)\cap \rmop{supp}(1-\chi^2)=\varnothing. 
    \]
    Then, by \eqref{eq_hwf_time_difference_2}, we obtain 
    \begin{align*}
        &A_\hbar (t_0)-A_\hbar (t_0)\chi\varphi^* a^\mathrm{w}(x, \hbar D)(\varphi_* \chi)\varphi_* \\
        &=A_\hbar (t_0)\Oph (c)+A_\hbar (t_0)(1-\chi^2)=O_{L^2\to L^2}(\hbar^\infty). \qedhere
    \end{align*}
\end{proof}

\begin{proof}[Proof of Proposition \ref{prop_wf_quantization}]
    Let $a\in C_c^\infty (T^*M)$ be a symbol such that $a=1$ near $(x_0, \xi_0)$ and $\Oph (a)u=O_{L^2}(\hbar^\infty)$. Take a coordinate function $\varphi: U\to V$ near $x_0$. We take a cutoff function $\chi\in C_c^\infty (U)$ and a symbol $b\in C_c^\infty (T^*\mathbb{R}^n)$ such that $\chi=1$ near $x_0$, $b=1$ near $\tilde\varphi (x_0, \xi_0)$ and $a=1$ near $\rmop{supp}\tilde\varphi^*b$. By Lemma \ref{lemm_psido_local_global}, there exists a symbol $c(\hbar; x, \xi)\in S^0_\mathrm{cyl}(T^*M)$ such that 
    \[
        \chi \varphi^* b^\mathrm{w}(x, \hbar D)(\varphi_*\chi)\varphi_*=\Oph (c)+O_{L^2\to L^2}(\hbar^\infty)
    \]
    and $\rmop{supp}c\subset \rmop{supp}\chi^2\tilde\varphi^*b$ modulo $O(\hbar^\infty)$. Since $\rmop{supp} \chi^2\tilde\varphi^*b \cap \rmop{supp}(1-a)=\varnothing$, Theorem \ref{theo_psido_composition} shows that $\Oph (c)\Oph (1-a)=O_{L^2\to L^2}(\hbar^\infty)$. Thus 
    \begin{align*}
        \chi \varphi^* b^\mathrm{w}(x, \hbar D)\varphi_*(\chi u)
        &=\Oph (c)u+O_{L^2}(\hbar^\infty) \\
        &=\Oph (c)\underbrace{\Oph (a)u}_{=O_{L^2}(\hbar^\infty)}+\underbrace{\Oph (c)\Oph (1-a)}_{=O_{L^2\to L^2}(\hbar^\infty)}u+O_{L^2}(\hbar^\infty) \\
        &=O_{L^2}(\hbar^\infty). 
    \end{align*}
    This shows that $(x_0, \xi_0)\not\in \rmop{WF}(u)$. 
\end{proof}

\subsection{Radially homogeneous wavefront sets and homogeneous wavefront sets}\label{subs_cylindrical_homogeneous}

In this section, we prove Proposition \ref{prop_hwf_polar} and Corollary \ref{coro_cylindrical_homogeneous}. 

\begin{proof}[Proof of Proposition \ref{prop_hwf_polar}]
    \fstep{(i) $\bm{\Rightarrow}$ (ii)}Assume that $(x_0, \xi_0)\not\in \rmop{HWF}(u)$. By definition of homogeneous wavefront sets, there exists a symbol $a\in C_c^\infty (T^*\mathbb{R}^n)$ such that $a=1$ near $(x_0, \xi_0)$ and $\| a^\mathrm{w}(\hbar x, \hbar D)u\|_{L^2}=O(\hbar^\infty)$. We can assume that $\rmop{supp}a \subset \Gamma \times \mathbb{R}^n$ for small conic neighborhood $\Gamma$ of $x_0$. Let $\varphi: \Gamma \to \mathbb{R}_+\times V^\prime$ be polar coordinates. Take a cylindrical function $\chi\in C^\infty (\mathbb{R}^n)$ such that $\rmop{supp}\chi \subset \Gamma$ and $\chi=1$ near $\rmop{supp}a$. Then, by the changing variables of pseudodifferential operators (see Section \ref{subs_cylindrical_class}), we have 
    \begin{equation}\label{eq_hwf_polar_change}
        (\chi-a)^\mathrm{w}(\hbar x, \hbar D)=\varphi^* b^\mathrm{w}(\hbar; r, \theta, \hbar D_r, \hbar D_\theta)\varphi_*, 
    \end{equation}
    where $b(\hbar; r, \theta, \rho, \eta)\in C_c^\infty (T^*\mathbb{R}^n)$ satisfies 
    \begin{align}
        \rmop{supp}b(\hbar)
        &\subset \{ \tilde\varphi (x, \xi)\in T^*\mathbb{R}^n \mid (\hbar x, \xi)\in \rmop{supp}(\chi-a)\} \nonumber\\
        &=\{ (r, \theta, \rho, \eta)\in T^*(\mathbb{R}_+\times V^\prime) \mid (\hbar r, \theta, \rho, \hbar \eta)\in \rmop{supp}\tilde\varphi_*(\chi-a)\} \label{eq_hwf_polar_wip}
    \end{align}
    modulo $O(\hbar^\infty)$. Here we employed the explicit form of $\tilde\varphi^{-1}$: 
    \begin{equation}\label{eq_lift_polar_explicit}
        \tilde\varphi^{-1}(r, \theta, \rho, \eta)=\left(r\omega (\theta), \rho \omega (\theta)+\frac{1}{r}\sum_{j, k=1}^{n-1}h^{jk}(\theta) \eta_j \frac{\partial \omega}{\partial \theta_k}(\theta)\right), 
    \end{equation}
    where $\varphi^{-1}(r, \theta)=r\omega(\theta)$, $\omega: V^\prime \to S^{n-1}$ is an embedding into the $(n-1)$-dimensional sphere $S^{n-1}$ and $(h^{jk}(\theta))_{j, k=1}^{n-1}$ is the inverse matrix of the positive definite symmetric matrix $(\partial_{\theta_j}\omega (\theta)\cdot \partial_{\theta_k}\omega(\theta))_{j, k=1}^{n-1}$ (equal to the metric tensor on the sphere). 
    
    We set $\tilde\varphi (x_0, \xi_0)=(r_0, \theta_0, \rho_0, \eta_0)$. Since $\chi-a=0$ near $(x_0, (\xi_0\cdot\hat x_0)\hat x_0)$, we can take a symbol $c(r, \theta, \rho, \eta)\in C_c^\infty (T^*(\mathbb{R}_+\times V^\prime))$ such that $\tilde\varphi_*(\chi-a)=0$ near the set 
    \[
        \{ (r, \theta, \rho, \eta) \mid (r, \theta, \rho, \eta+\eta_0)\in \rmop{supp}c \}. 
    \]
    Then 
    \[
    \rmop{supp} c(\hbar r, \theta, \rho, \hbar \eta)\cap \rmop{supp}\tilde\varphi_*(\chi-a)(\hbar r, \theta, \rho, \hbar \eta)=\varnothing. 
    \]
    Thus \eqref{eq_hwf_polar_change} implies
    \begin{equation}\label{eq_hwf_polar_change_2}
        c^\mathrm{w}(\hbar r, \theta, \hbar D_r, \hbar^2 D_\theta)b^\mathrm{w}(\hbar; r, \theta, \hbar D_r, \hbar D_\theta)=O_{L^2\to L^2}(\hbar^\infty). 
\end{equation}
Since 
\[
    \rmop{supp}c (\hbar r, \theta, \rho, \hbar\eta)\cap \rmop{supp}((\varphi_*\chi)(\hbar r, \theta, \rho, \eta)-\varphi_*\chi)=\varnothing,     
\]
\eqref{eq_hwf_polar_change_2} becomes 
\[
    c^\mathrm{w}(\hbar r, \theta, \hbar D_r, \hbar^2 D_\theta)(\varphi_*\chi)\varphi_*-\varphi_*a^\mathrm{w}(\hbar x, \hbar D)=O_{L^2\to L^2}(\hbar^\infty). 
\]
Hence, since $a^\mathrm{w}(\hbar x, \hbar D)u=O_{L^2}(\hbar^\infty)$, we have 
\[
    c^\mathrm{w}(\hbar r, \theta, \hbar D_r, \hbar^2 D_\theta)(\varphi_*(\chi u))=O_{L^2}(\hbar^\infty). 
\]

\fstep{(ii) $\bm{\Rightarrow}$ (i)}We take polar coordinates $\varphi: U\to V$, cylindrical function $\chi\in C^\infty (\mathbb{R}^n)$ and $a\in C_c^\infty (T^*\mathbb{R})$ as in the statement (ii). Take a cylindrical function $\chi\in C^\infty (\mathbb{R}^n)$ such that $\rmop{supp}\chi\subset U$ and $\chi=1$ near $\rmop{supp} \tilde\varphi^*a$. By the changing variables of pseudodifferential operators, we have 
\begin{equation}\label{eq_hwf_polar_change_inv}
    (\varphi_*\chi-a)^\mathrm{w}(\hbar r ,\theta, \hbar D_r, \hbar^2 D_\theta)=\varphi_* b^\mathrm{w}(\hbar; x, \hbar D)\varphi^*+O_{L^2\to L^2}(\hbar ^\infty), 
\end{equation}
where $b(\hbar; x, \xi)\in S^0(T^*\mathbb{R}^n)$ satisfies 
\begin{equation}\label{eq_hwf_polar_change_inv_wip}
    \begin{split}
        \rmop{supp}b(\hbar)
        &\subset \{ \tilde\varphi^{-1}(r, \theta, \rho, \eta) \mid (\hbar r, \theta, \rho, \hbar \eta)\in \rmop{supp}(\varphi_*\chi-a)\} \\
        &=\{ (x, \xi) \mid (\hbar x, \xi)\in \rmop{supp}(\chi-\tilde\varphi^* a)\}
    \end{split}
\end{equation}
modulo $O(\hbar^\infty)$ by \eqref{eq_lift_polar_explicit}. Thus we can take a symbol $c(x, \xi)\in C_c^\infty (T^*\mathbb{R}^n)$ such that $\chi-\tilde\varphi^*a=0$ near $\rmop{supp}c$. Then 
\[
    \rmop{supp}c(\hbar x, \xi)\cap \rmop{supp}b(\hbar; x, \xi)=\varnothing. 
\]
Thus \eqref{eq_hwf_polar_change_inv_wip} implies 
\begin{equation}\label{eq_hwf_polar_change_inv_2}
    c^\mathrm{w}(\hbar x, \hbar D)b^\mathrm{w}(\hbar; x, \hbar D)=O_{L^2\to L^2}(\hbar^\infty). 
\end{equation}
Since 
\[
    \rmop{supp}c (\hbar x, \xi)\cap \rmop{supp}(\chi(\hbar x))=\varnothing,     
\]
\eqref{eq_hwf_polar_change_inv_2} becomes 
\[
    c^\mathrm{w}(\hbar x, \hbar D)\varphi^*-\varphi^*a^\mathrm{w}(\hbar r, \theta, \hbar D_r, \hbar^2 D_\theta)=O_{L^2\to L^2}(\hbar^\infty). 
\]
Hence, since $a^\mathrm{w}(\hbar r, \theta, \hbar D_r, \hbar^2 D_\theta)\varphi_*(\chi u)=O_{L^2}(\hbar^\infty)$, we have 
\[
    c^\mathrm{w}(\hbar x, \hbar D)u=O_{L^2}(\hbar^\infty). \qedhere
\]
\end{proof}

\begin{proof}[Proof of Corollary \ref{coro_cylindrical_homogeneous}]
Assume that $x_0\neq 0$ and $(x_0, (\xi_0\cdot \hat x_0)\hat x_0)\not\in \rmop{HWF}(u)$, where $\hat x_0:=x_0/|x_0|$. By Proposition \ref{prop_hwf_polar}, there exist polar coordinates $\varphi: U\to V$, cylindrical function $\chi\in C^\infty (\mathbb{R}^n)$ with $\rmop{supp}\chi \subset U$ and $\chi=1$ near the set $\{ \lambda x \mid \lambda\geq 1\}$, and $a\in C_c^\infty (V)$ with $a=1$ near $\Psi (x, \xi)$ such that 
\begin{equation}\label{eq_hwf_polar_appl}
\|a^\mathrm{w}(\hbar r, \theta, \hbar D_r, \hbar^2 D_\theta)\varphi_*(\chi u)\|_{L^2(\mathbb{R}^n; \Omega^{1/2})}=O(\hbar^\infty)
\end{equation}
holds. 

Since $\tilde\varphi(x_0, (\xi_0\cdot\hat x_0)\hat x_0)=(|x_0|, \hat x_0, \xi_0\cdot \hat x_0, 0)$, the symbol $a$ is identically equals to 1 near $(|x_0|, \hat x_0, \xi_0\cdot \hat x_0, 0)$. Thus we can take a symbol $c(r, \theta, \rho, \eta)\in C_c^\infty (T^*(\mathbb{R}_+\times V^\prime))$ such that $a=1$ near the set 
\[
    \{ (r, \theta, \rho, \eta) \mid (r, \theta, \rho, \eta+\eta_0)\in \rmop{supp}c \}. 
\]
Then 
\[
    c^\mathrm{w}(\hbar r, \theta, \hbar D_r, \hbar D_\theta)a^\mathrm{w}(\hbar r, \theta, \hbar D_r, \hbar^2 D_\theta)
    =c^\mathrm{w}(\hbar r, \theta, \hbar D_r, \hbar D_\theta)+O_{L^2\to L^2}(\hbar^\infty)
\] 
for sufficiently small $\hbar>0$. Thus \eqref{eq_hwf_polar_appl} implies 
\[
    c^\mathrm{w}(\hbar r, \theta, \hbar D_r, \hbar D_\theta)\varphi_*(\chi u)=O_{L^2}(\hbar^\infty). \qedhere
\]
\end{proof}



\section{Estimates for Heisenberg derivatives}\label{sect_heisenberg_derivative}

\subsection{Estimates for symbols}

We begin with the estimate of $\psi_j=\varphi^*\tilde\psi_j$. Recall the definition \eqref{eq_defi_chi4} and \eqref{eq_defi_psi-1}. 

\begin{lemm}\label{lemm_estimate_tildepsi}
    For all multiindices $\alpha=(\alpha_0, \alpha^\prime), \beta=(\beta_0, \beta^\prime)\in \mathbb{Z}_{\geq 0}\times \mathbb{Z}_{\geq 0}^{n-1}$, the estimates
    \[ \| \partial_r^{\alpha_0}\partial_\theta^{\alpha^\prime} \partial_\rho^{\beta_0}\partial_\eta^{\beta^\prime} \tilde\psi_j(t)\|_{L^\infty} \leq C_{j\alpha\beta} t^{-\alpha_0}, \quad 
    \| \partial_r^{\alpha_0}\partial_\theta^{\alpha^\prime} \partial_\rho^{\beta_0}\partial_\eta^{\beta^\prime}\partial_t \tilde\psi_j(t)\|_{L^\infty} \leq C_{j\alpha\beta} t^{-\alpha_0}\]
    hold. 
\end{lemm}

\begin{rema*}
    We will only use the boundedness of derivatives of $\tilde\psi_j$, and the decay $t^{-\alpha_0}$ is not necessary for a proof of our main theorem. However, since $r\sim t$ on the support of $\tilde\psi_j$, Lemma \ref{lemm_estimate_tildepsi} states that $\partial_r^{\alpha_0}\partial_\theta^{\alpha^\prime} \partial_\rho^{\beta_0}\partial_\eta^{\beta^\prime} \tilde\psi_j(t)=O(r^{-\alpha_0})$ and $\partial_r^{\alpha_0}\partial_\theta^{\alpha^\prime} \partial_\rho^{\beta_0}\partial_\eta^{\beta^\prime}\partial_t \tilde\psi_j(t)=O(r^{-\alpha_0})$. 
\end{rema*}

\begin{proof}
    By the Leibnitz rule, it is enough to estimate each $\chi_{1j}, \chi_{2j}, \chi_{3j}, \chi_{4j}$ and their derivatives respectively. 
    
    \fstep{Estimate of $\bm{\chi_{1j}}$}A direct calculation shows that 
    \[
        |\partial_r^{\alpha_0} \chi_{1j}(t, r)|=|\chi^{(\alpha_0)}|(4\delta_j t)^{-\alpha_0}\leq \| \chi^{(\alpha_0)}\|_{L^\infty} (4\delta_j t)^{-\alpha_0}. 
    \]
    The time derivative of $\chi_{1j}$ is
    \[ \partial_t \chi_{1j}=\frac{1}{4\delta_j t} \chi^\prime \left(\frac{|r-r(t)|}{4\delta_j t}\right)\left(-\frac{\diff r}{\diff t}(t)\rmop{sgn} (r-r(t))-\frac{|r-r(t)|}{t}\right). \]
    The $r$ derivatives of the first term is estimated as 
    \begin{equation}\label{eq_r_derivative_tpsi} \left| \partial_r^{\alpha_0} \left(\chi^\prime \left(\frac{|r-r(t)|}{4\delta_j t}\right)\frac{\diff r}{\diff t}(t)\rmop{sgn} (r-r(t))\right)\right|\leq C\| \chi^{(\alpha_0+1)}\|_{L^\infty} (4\delta_j t)^{-\alpha_0}
    \end{equation}
    by the Hamilton equation \eqref{eq_hamilton_equation_radial} and the boundedness of $|\rho(t)|$ insured by Theorem \ref{theo_classical_estimate}. The second term of \eqref{eq_r_derivative_tpsi} is written as 
    \[
        \chi^\prime \left(\frac{|r-r(t)|}{4\delta_j t}\right)\frac{|r-r(t)|}{t}=4\delta_j \tilde \chi \left(\frac{r-r(t)}{4\delta_j t}\right), 
    \]
        where $\tilde \chi (x):=|x|\chi^\prime (|x|)\in C_c^\infty(\mathbb{R})$. 
        Thus a similar estimate to \eqref{eq_r_derivative_tpsi} shows that 
    \[ 
        \left|\partial_r^{\alpha_0} \left(\chi^\prime \left(\frac{|r-r(t)|}{4\delta_j t}\right)\frac{|r-r(t)|}{t}\right)\right|\leq \| \tilde \chi \|_{L^\infty} (4\delta_j)^{-\alpha_0+1} t^{-\alpha_0}.
    \]
    Hence if $0<\delta_j\leq 1/4$, then $|\partial_r^{\alpha_0}\partial_t \chi_{1j}|\leq C_{j\alpha_0} t^{-\alpha_0-1}$. 

    \fstep{Estimate of $\bm{\partial^\alpha\chi_{2j}}$}A similar estimate to $\chi_{1j}$ shows 
    \[ |\partial_\theta^{\alpha^\prime} \chi_{2j}(t, \theta)|\leq \| \partial_\theta^{\alpha^\prime} \chi\|_{L^\infty} (\delta_j-t^{-\lambda})^{-|\alpha^\prime|}. \]

    The time derivative of $\chi_{2j}$ is 
    \begin{equation}\label{eq_t_derivative_chi2}
        \begin{split}
            &\partial_t \chi_{2j}(t, \theta) \\
            &=\frac{1}{\delta_j-t^{-\lambda}} \chi^\prime \left(\frac{|\theta-\theta(t)|}{\delta_j-t^{-\lambda}}\right)\left(-\frac{\diff \theta}{\diff t}(t)\cdot\frac{\theta-\theta(t)}{|\theta-\theta(t)|}-\frac{\lambda t^{-\lambda-1}|\theta-\theta(t)|}{\delta_j-t^{-\lambda}}\right). 
        \end{split}
    \end{equation}
    We set $F_k(x):=x_k\chi^\prime (|x|)/|x|\in C_c^\infty(\mathbb{R}^{n-1})$. Then we have 
    \[\chi^\prime \left(\frac{|\theta-\theta(t)|}{\delta_j-t^{-\lambda}}\right)\frac{\diff \theta}{\diff t}(t)\cdot\frac{\theta-\theta(t)}{|\theta-\theta(t)|}=\sum_{k, l=1}^{n-1}h^{kl}(r(t), \theta (t))F_k\left(\frac{\theta-\theta(t)}{\delta_j-t^{-\lambda}}\right)\eta_l(t)
    \]
    by the Hamilton equation \eqref{eq_hamilton_equation_angle}. We apply the boundedness of $|\eta(t)|$ with respect to the fiber metric $h^*(1, \theta, \partial_\theta)$ by Theorem \ref{theo_classical_estimate} and $h^*(r, \theta, \eta)\leq Ch^*(1, \theta, \eta)$ by \eqref{eq_ineq_f_logbdd} and \eqref{eq_ineq_model_bdd}. Then we obtain 
    \[
        \left|\partial_\theta^{\alpha^\prime}\left(\sum_{k, l=1}^{n-1}h^{kl}(r(t), \theta (t))F_k\left(\frac{\theta-\theta(t)}{\delta_j-t^{-\lambda}}\right)\eta_l(t)\right)\right|
        \leq C_{j\alpha^\prime}. 
    \]
    For the second term of \eqref{eq_t_derivative_chi2}, if we set $\tilde \chi (x)=|x|\chi^\prime (|x|)\in C_c^\infty(\mathbb{R}^{n-1})$, then 
    \begin{align*}
        \left|\partial_\theta^{\alpha^\prime} \left(\chi^\prime \left(\frac{|\theta-\theta(t)|}{\delta_j-t^{-\lambda}}\right)\frac{|\theta-\theta(t)|}{\delta_j-t^{-\lambda}}\right)\right|
    =\left|\partial_\theta^{\alpha^\prime} \left( \tilde\chi \left(\frac{\theta-\theta(t)}{\delta_j-t^{-\lambda}}\right)\right)\right| 
    \leq C_{j\alpha^\prime}. 
    \end{align*}
    Hence the $\theta$ derivative of \eqref{eq_t_derivative_chi2} is estimated as $|\partial_\theta^{\alpha^\prime}\partial_t \chi_{2j}|\leq C_{j\alpha^\prime}$. 

    \fstep{Estimate of $\bm{\partial^\alpha\chi_{3j}}$}By the same procedure as the estimate of $\chi_{2j}$, we have $|\partial_\rho^{\beta_0} \partial_t^a \chi_{3j}|\leq C_{j\beta_0}$ ($\beta_0\geq 0$, $a=0, 1$). 

    \fstep{Estimate of $\bm{\partial^\alpha\chi_{4j}}$}We have $|\partial_\eta^{\beta^\prime}\chi_{4j}|\leq C_{j\beta^\prime}$ by the same procedure as the estimate of $\partial_\theta^{\alpha^\prime}\chi_{2j}$. The $t$ derivative is 
    \begin{equation}\label{eq_t_derivative_chi4}
        \partial_t \chi_{4j}=\frac{1}{\delta_j-t^{-\lambda}} \chi^\prime \left(\frac{|\eta-\eta(t)|}{\delta_j-t^{-\lambda}}\right)\left(-\frac{\diff \eta}{\diff t}(t)\cdot\frac{\eta-\eta(t)}{|\eta-\eta(t)|}-\lambda t^{-\lambda-1}|\eta-\eta(t)|\right). 
    \end{equation}
    The $\eta$ derivative of the first term is estimated as 
    \begin{align*}
        &\left|\partial_\eta^{\beta^\prime}\left(\chi^\prime \left(\frac{|\eta-\eta(t)|}{\delta_j-t^{-\lambda}}\right)\frac{\diff \eta}{\diff t}(t)\cdot\frac{\eta-\eta(t)}{|\eta-\eta(t)|}\right)\right| \\
    &=\left|\frac{\diff \eta}{\diff t} (t)\cdot\partial_\eta^{\beta^\prime}\left(F\left(\frac{\eta-\eta(t)}{\delta_j-t^{-\lambda}}\right)\right)\right|\leq C\left|\frac{\diff \eta}{\diff t} (t)\right|.
\end{align*}
    By the angular momentum component of Hamilton equations
    \begin{equation}\label{eq_hamilton_equation_angular_momentum}
        \frac{\diff \eta_j}{\diff t} (t)=-\frac{1}{2}\frac{\partial h^{kl}}{\partial \theta_j}(r(t), \theta(t))\eta_k(t)\eta_l(t)
    \end{equation}
    and $|\eta(t)|\leq C$ by Theorem \ref{theo_classical_estimate}, we obtain 
    \[ \left|\partial_\eta^{\beta^\prime}\left(\chi^\prime \left(\frac{|\eta-\eta(t)|}{\delta_j -^{-\lambda} }\right)\frac{d \eta}{\diff t}(t)\cdot\frac{\eta-\eta(t)}{|\eta-\eta(t)|}\right)\right|\leq C_{j\beta^\prime}.  \]
    The $\eta$ derivatives of the second term in \eqref{eq_t_derivative_chi4} are estimated as 
    \begin{align*}
        &\left|\partial_\eta^{\beta^\prime}\left(\chi^\prime \left(\frac{|\eta-\eta(t)|}{\delta_j-t^{-\lambda}}\right)|\eta-\eta(t)|\right)\right| \\
        &=(\delta_j-t^{-\lambda})  \left|\partial_\eta^{\beta^\prime} \left(\tilde\chi \left(\frac{\eta-\eta(t)}{\delta_j-t^{-\lambda}}\right)\right)\right|
    \leq C_{j\beta^\prime}. 
    \end{align*}
    Hence the derivatives of \eqref{eq_t_derivative_chi4} are estimated as $|\partial_\eta^{\beta^\prime}\partial_t \chi_{4j}|\leq C_{j\beta^\prime}$. 
\end{proof}

Next we prove the positivity and an $O(\jbracket{t}^{-1})$ decay as $t\to\infty$ of the Lagrange derivative $\partial_t \tilde\psi_j(t)+\{ \tilde\psi_j(t), h_0\}$. Both of them play a crucial role in estimates of Heisenberg derivatives in the proof of Theorem \ref{theo_positive_hd}.  

\begin{lemm}\label{lemm_positive_lagrange_derivative}
    Take sufficiently small $0<\delta_0<\delta_1<\cdots <2\delta_0$ and $0<\lambda<\min\{2c_0-1, \mu\}$ in \eqref{eq_defi_chi4}. Then the inequalities
    \[
        \partial_t \tilde\psi_j(t)+\{ \tilde\psi_j(t), h_0\}\geq 0
    \]
    and
    \[
        |\partial_r^{\alpha_0}\partial_\theta^{\alpha^\prime} \partial_\rho^{\beta_0}\partial_\eta^{\beta^\prime}(\partial_t \tilde\psi_j(t)+\{ \tilde\psi_j(t), h_0\})|\leq C_{j\alpha\beta}\jbracket{t}^{-1}
    \]
    hold for sufficiently large $t>0$. 
\end{lemm}

\begin{proof}
    As in the proof of Lemma \ref{lemm_estimate_tildepsi}, we estimate each $\partial_t\chi_{jk}+\{ \chi_{jk}, h_0\}$ respectively. We also borrow following functions from the proof of Lemma \ref{lemm_estimate_tildepsi}: 
    \[
        \tilde \chi (x):=|x|\chi^\prime (|x|)\in C_c^\infty(\mathbb{R}), \quad F_k(x):=x_k\chi^\prime (|x|)/|x|\in C_c^\infty(\mathbb{R}^{n-1}). 
    \]

    \fstep{$\bm{\partial_t \chi_{1j}+\{ \chi_{1j}, h_0\}}$}We have
    \begin{align}
        &\partial_t \chi_{1j} +\{ \chi_{1j}, h_0\} \nonumber\\
        &=\frac{|\chi^\prime|}{2\delta_j t} \left(\frac{|r-r(t)|}{t}+(c(r(t), \theta(t))^{-2}\rho(t)-c(r, \theta)^{-2}\rho)\rmop{sgn}(r-r(t))\right) \label{eq_lagrange_derivative_11}\\
        &=-\frac{1}{t} \tilde\chi \left(\frac{r-r(t)}{2\delta_j t}\right)-\frac{c(r(t), \theta(t))^{-2}\rho(t)-c(r, \theta)^{-2}\rho}{2\delta_j t} F\left(\frac{r-r(t)}{2\delta_j t}\right). \label{eq_lagrange_derivative_12}
    \end{align}
    Here $F(x):=\chi^\prime (|x|)\rmop{sgn}x\in C_c^\infty (\mathbb{R})$. We recall the short range condition $c(r, \theta)=1+O(r^{-1-\mu})$ (Assumption \ref{assu_classical} \ref{assu_sub_short_range}). Since $|r-r(t)|\geq 4\delta_j t$ and $|\rho-\rho(t)|\leq 2\delta_j-2$ on the support of $(\partial\chi_{1j})\chi_{2j} \chi_{3j} \chi_{4j}$, we have
    \begin{align*}
        &|c(r(t), \theta(t))^{-2}\rho(t)-c(r, \theta)^{-2}\rho| \\
        &\leq c(r(t), \theta(t))^{-2}|\rho (t)-\rho|+|\rho||c(r(t), \theta(t))^{-2}-c(r, \theta)^{-2}| \\
        &\leq 2(1+Cr(t)^{-1-\mu})(\delta_j-)+C(1+\delta_j-)(r^{-1-\mu}+r(t)^{-1-\mu}) \\
        &\leq 2\delta_j+Ct^{-\lambda}, 
    \end{align*}
    and thus, by \eqref{eq_lagrange_derivative_11}, 
    \begin{align*}
        \partial_t \chi_{1j} +\{ \chi_{1j}, h_0\}
        \geq \frac{|\chi^\prime|}{2\delta_j t} \left(\frac{|r-r(t)|}{t}-|\rho(t)-\rho|\right)
        \geq \frac{|\chi^\prime|}{2\delta_j t}(4\delta_j-2\delta_j-2)\geq 0. 
    \end{align*}
    $r$ derivatives are estimated as 
        \begin{align*} 
        &|\partial_r^{\alpha_0} \partial_\rho^{\beta_0} (\partial_t \chi_{1j}(t, r) +\{ \chi_{1j}, h_0\})(t, r, \rho)| \\ 
        &\leq     
        \begin{cases}
            C(1+|\delta_j-t^{-\lambda}|)\delta_j^{-\alpha_0} t^{-\alpha_0-1} & \text{if } \beta_0=0, \\
            C(\delta_j t)^{-\alpha_0-1} & \text{if } \beta_0=1, \\
            0 & \text{if } \beta_0=2
        \end{cases} \\
    &\leq C_{j\alpha_0\beta_0}t^{-\alpha_0-1}
\end{align*}
by \eqref{eq_lagrange_derivative_12}. 

\fstep{$\bm{\partial_t \chi_{2j}+\{ \chi_{2j}, h_0\}}$}By the Hamilton equation \eqref{eq_hamilton_equation_angle}, we have
    \begin{align}
        &\partial_t \chi_{2j}+\{\chi_{2j}, h_0\} \nonumber\\
        &\begin{aligned}
        &=\frac{|\chi^\prime|}{\delta_j-t^{-\lambda}} \biggl(\frac{\lambda |\theta-\theta (t)|}{\delta_j-t^{-\lambda}} \\
        &\quad+\sum_{k, l=1}^{n-1}(h^{kl}(r(t), \theta (t))\eta_k (t)-h^{kl}(r, \theta)\eta_k)\frac{\theta_l -\theta_l (t)}{|\theta-\theta (t)|}\biggr)
        \end{aligned} \label{eq_lagrange_derivative_21}\\
        &
        \begin{aligned}
        &=-\frac{\lambda }{\delta_j-t^{-\lambda}} \tilde\chi \left(\frac{\theta-\theta(t)}{\delta_j-t^{-\lambda}}\right) \\
        &\quad -\frac{1}{\delta_j-t^{-\lambda}}\sum_{k, l=1}^{n-1}(h^{kl}(r(t), \theta (t))\eta_k (t)-h^{kl}(r, \theta)\eta_k) F_l\left(\frac{\theta-\theta(t)}{\delta_j-t^{-\lambda}}\right). 
        \end{aligned}\label{eq_lagrange_derivative_22}
    \end{align}
    Since $|\theta-\theta (t)|\geq \delta_j -$, $|r-r(t)|\leq 8\delta_jt$ and $|\eta-\eta(t)|\leq 2(\delta_j-t^{-\lambda})\leq 2\delta_j$ on the support of $\chi_{1j}(\partial \chi_{2j})\chi_{3j} \chi_{4j}$, we obtain the following inequality from \eqref{eq_lagrange_derivative_21}: 
    \begin{align*}
        &|h^*(r(t), \theta (t))\eta (t)-h^*(r, \theta)\eta| \\
        &\leq |h^*(r(t), \theta (t))(\eta (t)-\eta)|+|(h^*(r(t), \theta(t))-h^*(r, \theta))\eta| \\
        &\leq Cf(r(t))^{-2}+C(f(r(t))^{-2}+f(r)^{-2}) \\
        &\leq Cf(r(t))^{-2}
    \end{align*}
    and thus 
    \begin{align*}
        &\partial_t \chi_{2j}+\{\chi_{2j}, h_0\} \\
        &\geq \frac{|\chi^\prime|}{\delta_j-t^{-\lambda}} \left(\lambda -|h^*(r(t), \theta (t))\eta (t)-h^*(r, \theta)\eta|\right) \\
        &\geq \frac{|\chi^\prime|}{\delta_j-t^{-\lambda}} (\lambda -Cf(r(t))^{-2})\geq 0. 
    \end{align*}    
    Here we introduced a shorthand notation $\left(\sum_{l=1}^{n-1}h^{kl}(r, \theta)\eta_l\right)_{k=1}^{n-1}=h^*(r, \theta)\eta$. 
    We employ the estimates of classical orbits $r(t)\geq \rho_\infty t-C$, $|\eta(t)|\leq C$ by Theorem \ref{theo_classical_estimate}, $f(r)\geq C^{-1}r^{2c_0}$ by \eqref{eq_ineq_f_logbdd} and the assumption $0<\lambda<2c_0-1$, and we obtain
    \[ 
        \partial_t \chi_{2j}+\{\chi_{2j}, h_0\}\geq \frac{C}{\delta_j-t^{-\lambda}} (\lambda -Ct^{-2c_0})\geq 0. 
    \]
    Furthermore we have 
    \[ |\partial_r^{\alpha_0}\partial_\theta^{\alpha^\prime} \partial_\eta^{\beta^\prime}(\partial_t \chi_{2j}+\{\chi_{2j}, h_0\})|
        \leq C_{j\alpha\beta^\prime}t^{-1-\lambda}
\]
by differentiating \eqref{eq_lagrange_derivative_22}. 

\fstep{$\bm{\partial_t \chi_{3j}+\{ \chi_{3j}, h_0\}}$}We have 
\begin{align*}
    &\partial_t \chi_{3j} +\{ \chi_{3j}, h_0\} \\
    &=\frac{|\chi^\prime|}{\delta_j-t^{-\lambda}} \biggl(\frac{\lambda |\rho-\rho (t)|}{\delta_j-t^{-\lambda}} \\
    &\quad +(\partial_r h_0(r(t), \theta (t), \rho(t), \eta(t))-\partial_r h_0(r, \theta, \rho, \eta)) \frac{\rho -\rho (t)}{|\rho-\rho (t)|}\biggr) \\
    &=-\frac{\lambda }{\delta_j-t^{-\lambda}} \tilde\chi \left(\frac{\rho-\rho (t)}{\delta_j-t^{-\lambda}}\right) \\
    &\quad -\frac{\partial_r h_0(r(t), \theta (t), \rho(t), \eta(t))-\partial_r h_0(r, \theta, \rho, \eta)}{\delta_j-t^{-\lambda}} F\left(\frac{\rho-\rho (t)}{\delta_j-t^{-\lambda}}\right). 
\end{align*}
As in the estimate of $\partial_t \chi_{2j}+\{ \chi_{2j}, h_0\}$, on the support of $\chi_{1j} \chi_{2j} (\partial \chi_{3j})\chi_{4j}$, we have the estimate
\begin{align*}
    &|\partial_r h_0(r(t), \theta (t), \rho(t), \eta(t))-\partial_r h_0(r, \theta, \rho, \eta)| \\
    &\leq |\rho^2\partial_r c^{-2}/2|+|\partial_r h^*(r, \theta, \eta)| \\
    &\leq C(r^{-1-\mu}\rho^2+r^{-2c_0}|\eta|^2)
\end{align*}
by Assumption \ref{assu_higher_derivative} and thus 
\[
    \partial_t \chi_{3j}+\{ \chi_{3j}, h_0\}
    \geq \frac{|\chi^\prime|}{\delta_j-t^{-\lambda}} \left(\lambda -C(r^{-1-\mu}+r(t)^{-1-\mu})\right)\geq 0 
    \]
    and 
    \[ |\partial_r^{\alpha_0}\partial_\theta^{\alpha^\prime} \partial_\rho^{\beta_0}\partial_\eta^{\beta^\prime} (\partial_t \chi_{3j}+\{ \chi_{3j}, h_0\})|\leq C_{j\alpha\beta}(\lambda +t^{-1-\mu})\leq C_{j\alpha\beta}t^{-1-\lambda}. 
    \]

    \fstep{$\bm{\partial_t \chi_{4j}+\{ \chi_{4j}, h_0\}}$}By the Hamilton equation \eqref{eq_hamilton_equation_angular_momentum}, we have
\begin{align*}
    &\partial_t \chi_{4j}+\{ \chi_{4j}, h_0\} \\
    &=\frac{|\chi^\prime|}{\delta_j-t^{-\lambda}} \biggl(\frac{\lambda |\eta-\eta (t)|}{\delta_j-t^{-\lambda}} \\
    &\quad -\sum_{k=1}^{n-1} (\partial_{\theta_k}h_0(r(t), \theta (t), \rho(t), \eta(t))-\partial_{\theta_k}h_0(r, \theta, \rho, \eta))\frac{\eta_k -\eta_k (t)}{|\eta-\eta (t)|}\biggr) \\
    &=-\frac{\lambda }{\delta_j-t^{-\lambda}}\tilde\chi \left(\frac{\eta-\eta (t)}{\delta_j- t^{-\lambda}}\right) \\
    &\quad-\sum_{k=1}^{n-1}\frac{\partial_{\theta_k}h_0(r(t), \theta (t), \rho(t), \eta(t))-\partial_{\theta_k}h_0(r, \theta, \rho, \eta)}{\delta_j-t^{-\lambda}}  F_k\left(\frac{\eta-\eta (t)}{\delta_j-t^{-\lambda}}\right). 
\end{align*}
On the support of $\chi_{1j} \chi_{2j} \chi_{3j} (\partial \chi_{4j})$, we have
\[
    |\partial_{\theta_k}h_0(r, \theta, \rho, \eta)|
    \leq |\rho^2\partial_{\theta_k}c^{-2}|+|\partial_{\theta_k}h^*(r, \theta, \eta)|\leq C(r^{-1-\mu}\rho^2+r^{-2c_0}|\eta|^2)
\]
and thus 
\[
    \partial_t \chi_{4j}+\{ \chi_{4j}, h_0\}
    \geq \frac{|\chi^\prime|}{\delta_j-t^{-\lambda}} \left(\lambda -C(t^{-1-\mu}+t^{-2c_0})\right)\geq 0 
    \]
    and 
    \[ |\partial_r^{\alpha_0}\partial_\theta^{\alpha^\prime} \partial_\rho^{\beta_0}\partial_\eta^{\beta^\prime}  (\partial_t \chi_{4j}+\{ \chi_{4j}, h_0\})|\leq C_{j\alpha\beta}t^{-\lambda-1}. \qedhere 
    \]
\end{proof}

\subsection{Proof of Theorem \ref{theo_symbol_aim}}

We prove Theorem \ref{theo_symbol_aim} in this section. 

\begin{theo}\label{theo_positive_hd}
    There exist constants $c_1, c_2, c_3\ldots>0$ such that, if we set 
    \[
        F_k(t):=\Oph (\psi_0(t))^*\Oph(\psi_0(t))+t\sum_{j=1}^k c_j \hbar^j \Oph (\psi_j(t)), 
    \]
    for $k\geq 1$, then the inequality 
    \begin{equation}\label{eq_a0_from_below}
        \partial_t F_k(t)-i\hbar [F_k(t), H] 
         \geq O_{L^2\to L^2}(\hbar^{k+1})
    \end{equation}
    holds for all $\hbar\in (0, 1]$ uniformly in $t\geq 0$. 
\end{theo}

\begin{proof}
    \initstep\step We first prove the existence of a real symbol $b_0(\hbar; t, x, y)\in S^0_\mathrm{cyl}(T^*M)$ which satisfies 
    \begin{equation}\label{eq_hd_step1}
        \begin{split}
            &\partial_t (\Oph (\psi_0(t))^*\Oph (\psi_0(t)))-i\hbar [\Oph(\psi_0(t))^*\Oph(\psi_0(t)), H] \\
        &\geq -\hbar \Oph (b_0(t))+O_{L^2\to L^2}(\hbar^\infty)
        \end{split}
    \end{equation}
    and has an asymptotic expansion 
    \begin{equation}\label{eq_asymptotic_hd_step1}
        b_0(\hbar; t, x, \xi)\sim \sum_{j=0}^\infty \hbar^j b_{0j}(t, x, \xi), \quad b_{0j}(t, x, \xi)\in S^{-j}_\mathrm{cyl}(T^*M)
    \end{equation}
    with $\rmop{supp}b_{0j}(t)\subset \rmop{supp}\psi_0(t)$. 

    By the Leibnitz rule, we have 
    \begin{align*}
            &\partial_t (\Oph (\psi_0(t))^*\Oph (\psi_0(t)))-i\hbar [\Oph(\psi_0(t))^*\Oph(\psi_0(t)), H] \\
        &=2\rmop{Re} \Oph (\psi_0(t))^*(\partial_t \Oph (\psi_0(t))-i\hbar [\Oph(\psi_0(t)), H]). 
    \end{align*}
    We employ Theorem \ref{theo_psido_composition} and Theorem \ref{theo_laplacian_psido}, and we take a symbol $b(\hbar; t, x, \xi)\in S^0_\mathrm{cyl}(T^*M)$ which satisfies 
    \begin{align*}
            &2\rmop{Re} \Oph (\psi_0(t))^*(\partial_t \Oph (\psi_0(t))-i\hbar^{-1} [\Oph(\psi_0(t)), H]) \\
            &=\Oph \left(\partial_t |\psi_0(t)|^2+\{|\psi_0(t)|^2, h_0\}+\hbar b(t)\right)+O_{L^2\to L^2}(\hbar^\infty)
    \end{align*}
    and has an asymptotic expansion 
    \[
        b(\hbar; t, x, \xi)\sim \sum_{j=0}^\infty \hbar^j b_j(t, x, \xi), \quad b_j(t, x, \xi)\in S^{-j}_\mathrm{cyl}(T^*M)
    \]
    with $\rmop{supp}b_j(t)\subset \rmop{supp}\psi_0(t)$. Since $\partial_t |\psi_0(t)|^2+\{ |\psi_0(t)|^2, h_0\}\in S^0_\mathrm{cyl}(T^*M)$ and $\partial_t |\psi_0(t)|^2+\{ |\psi_0(t)|^2, h_0\}\geq 0$, we apply Theorem \ref{theo_sharp_garding} and obtain 
        \[
            \Oph \left(\partial_t |\psi_0(t)|^2+\{|\psi_0(t)|^2, h_0\}\right)
            \geq -\hbar \Oph (c)+O_{L^2\to L^2}(\hbar^\infty), 
        \]
        where $c=c(\hbar; t, x, \xi)$ has an asymptotic expansion 
        \[
            c(\hbar; t, x, \xi)\sim \sum_{j=0}^\infty \hbar^j c_j(t, x, \xi), \quad c_j(t, x, \xi)\in S^{-j}_\mathrm{cyl}(T^*M)
        \]
        with $\rmop{supp}c_j (t)\subset \rmop{supp}\psi_0(t)$. We obtain \eqref{eq_hd_step1} and \eqref{eq_asymptotic_hd_step1} by setting $b_0=b+c$ and $b_{0j}=b_j+c_j$. 

    \step We secondly prove that, if we take a sufficiently large constant $c_1>0$ and set $F_1(t):=\Oph (\psi_0(t))^*\Oph (\psi_0(t))+c_1\hbar t \Oph (a_1(t))$, then we have 
    \begin{equation}
        \label{eq_hd_step2}
        \partial_t F_1(t)-i\hbar [F_1(t), H]\geq -\hbar^2 \Oph (b_1(t))+O_{L^2\to L^2}(\hbar^\infty), 
    \end{equation}
    where $b_1=b_1(\hbar; t, x, \xi)\in S^0_\mathrm{cyl}(T^*M)$ has an asymptotic expansion 
    \begin{equation}\label{eq_asymptotic_hd_step2}
        b_1(\hbar; t, x, \xi)\sim \sum_{j=0}^\infty \hbar^j b_{1j}(t, x, \xi), \quad b_{1j}(t, x, \xi)\in S^{-j}_\mathrm{cyl}(T^*M)
    \end{equation}
    with $\rmop{supp}b_{1j} (t)\subset \rmop{supp}a_1(t)$. 

    The left hand side of \eqref{eq_hd_step2} is equal to 
    \begin{align*}
        &\partial_t (\Oph (\psi_0(t))^*\Oph (\psi_0(t)))-i\hbar [\Oph(\psi_0(t))^*\Oph(\psi_0(t)), H] \\
        &+c_1\hbar t (\partial_t \Oph (a_1(t))-i\hbar [\Oph(a_1(t)), H]) \\
        &+c_1\hbar \Oph (a_1(t)). 
    \end{align*}
    The first term is estimated by \eqref{eq_hd_step1}. Since $\partial_t a_1(t)+\{ a_1(t), h_0\}=O_{S^0_\mathrm{cyl}(T^*M)}(\jbracket{t}^{-1})$ and $\partial_t a_1(t)+\{ a_1(t), h_0\}\geq 0$ by Lemma \ref{lemm_positive_lagrange_derivative}, we apply Theorem \ref{theo_sharp_garding} for the second term and obtain a symbol $b(\hbar; t, x, \xi)\in S^0_\mathrm{cyl}(T^*M)$ which satisfies 
    \[
            \Oph \left(\partial_t a_1(t)+\{a_1(t), h_0\}\right)
            \geq -\hbar \jbracket{t}^{-1}\Oph (b^\prime (t))+O_{L^2\to L^2}(\hbar^\infty)
        \]
        and has an asymptotic expansion 
        \[
            b^\prime (\hbar; t, x, \xi)\sim \sum_{j=0}^\infty \hbar^j b^\prime_j(t, x, \xi), \quad b_j(t, x, \xi)\in S^{-j}_\mathrm{cyl}(T^*M)
        \]
        with $\rmop{supp}b^\prime_j (t)\subset \rmop{supp}a_1(t)$. Hence we have 
    \begin{equation}
       \label{eq_hd_step2_wip}
        \begin{split}
            &\partial_t F_1(t)-i\hbar [F_1(t), H] \\
            &\geq 
            -\hbar \Oph (b_0(t))-c_1\hbar^2 t\jbracket{t}^{-1}\Oph (b^\prime (t))+c_1\hbar \Oph (a_1(t))
            +O_{L^2\to L^2}(\hbar^\infty).
            \end{split} 
    \end{equation}
    Since $\rmop{supp} b_0(t)\subset \rmop{supp}\psi_0(t)$ mod $O(\hbar^\infty)$ by \eqref{eq_asymptotic_hd_step1} and $a_1(t)=1$ near $\rmop{supp}\psi_0(t)$, we can take a constant $c_1>0$ such that 
    \begin{equation}\label{eq_hd_step2_minor}
        - \Oph (b_0(t))+c_1 \Oph (a_1(t)) 
        \geq -\hbar \Oph (b^\pprime (t))+O_{L^2\to L^2}(\hbar^\infty)
    \end{equation}
    where $b^\pprime (\hbar; t, x, \xi)\in S^0_\mathrm{cyl}(T^*M)$ has an asymptotic expansion 
    \[
            b^\pprime (\hbar; t, x, \xi)\sim \sum_{j=0}^\infty \hbar^j b^\pprime_j(t, x, \xi), \quad b^\pprime_j(t, x, \xi)\in S^{-j}_\mathrm{cyl}(T^*M)
    \]
    with $\rmop{supp}b^\pprime_j (t)\subset \rmop{supp}a_1(t)$. We set $b_1(t):=b^\pprime (t)+c_1t\jbracket{t}^{-1}b^\prime (t)$. Then \eqref{eq_hd_step2_wip} and \eqref{eq_hd_step2_minor} implies \eqref{eq_hd_step2} and \eqref{eq_asymptotic_hd_step2} with $b_{1j}(t):=b^\pprime_j (t)+c_1t\jbracket{t}^{-1}b^\prime_j (t)$. 

    \step We repeat the procedure in Step 2 and obtain positive constants $c_2, c_3, \ldots >0$ and $b_k(\hbar; t, x, \xi)\in S^0_\mathrm{cyl}(T^*M)$ such that, if we set 
    \[
        F_k(t):=\Oph (\psi_0(t))^*\Oph(\psi_0(t))+t\sum_{j=1}^k c_j \hbar^j \Oph (\psi_j(t)), 
    \]
    then 
    \[
        \partial_t F_k(t)-i\hbar [F_k(t), H]\geq -\hbar^{k+1} \Oph (b_k(t))+O_{L^2\to L^2}(\hbar^\infty)
    \]
    and $b_k=b_k(\hbar; t, x, \xi)\in S^0_\mathrm{cyl}(T^*M)$ has an asymptotic expansion 
    \[
        b_k(\hbar; t, x, \xi)\sim \sum_{j=0}^\infty \hbar^j b_{kj}(t, x, \xi), \quad b_{kj}(t, x, \xi)\in S^{-j}_\mathrm{cyl}(T^*M)
    \]
    with $\rmop{supp}b_{kj} (t)\subset \rmop{supp}a_k(t)$. In particular, since $\|\Oph (b_k(t))\|_{L^2\to L^2}$ is uniformly bounded in $t\geq 0$ and $0<\hbar \leq 1$, we obtain the desired inequality \eqref{eq_a0_from_below}. 
\end{proof}

\begin{proof}[Proof of Theorem \ref{theo_symbol_aim}]
    (i) is an immediate consequence of Lemma \ref{lemm_estimate_tildepsi} and the definition \eqref{eq_transport} of $\psi_j(t, x, \xi)$. 

    (ii) Take symbols $\psi_j(t)$ in Theorem \ref{theo_positive_hd} and define $\tilde a(\hbar; t, x, \xi)\in S^{-2}_\mathrm{cyl}(T^*M)$ by an asymptotic expansion 
    \[
        \tilde a(\hbar; t, x, \xi)\sim \sum_{j=1}^\infty c_{j-1} \hbar^j \psi_j(\hbar^{-1}t). 
    \]
    We set 
    \[
        A_\hbar (t):=\Oph (a_0(\hbar^{-1}t))^*\Oph(a_0(\hbar^{-1}t))+t\Oph (\tilde a(\hbar; t))
    \]
    Then, by Theorem \ref{theo_positive_hd}, we have 
    \begin{align*}
        \partial_t A_\hbar (t)-i[A_\hbar (t), H]
        &=\hbar^{-1}(\partial_t F_k(\hbar^{-1}t)-i\hbar [F_k(\hbar^{-1}t), H])+O_{L^2\to L^2}(\hbar^k) \\
        &\geq O_{L^2\to L^2}(\hbar^k)
    \end{align*}
    for all $k\geq 0$ uniformly in $t\in [0, t_0]$. 
\end{proof}

\appendix
\section{Escape functions}\label{subs_escape_function}

In this appendix, we construct a diffeomorphism $\Psi: E \to \mathbb{R}_+ \times S$ in Assumption \ref{assu_manifold_with_end_0} by employing an escape function. In this paper, we employ the terminology ``escape function'' in the following sense. 

\begin{defi}\label{defi_escape_function}
    A continuous function $r\in C (M; [0, \infty))$ on $M$ is an \textit{escape function} if 
    \begin{enumerate}
        \renewcommand{\labelenumi}{(\roman{enumi})}
        \item $r(M)=[0, \infty)$; 
        \item the preimage $r^{-1}([0, R])$ is compact for all $R\geq 0$; 
        \item $r(x)$ is $C^\infty$ in $r^{-1}(\mathbb{R}_+)$ and $\diff r(x)\neq 0$ for all $x\in r^{-1}(\mathbb{R}_+)$. Here $\mathbb{R}_+:=(0, \infty)$. 
    \end{enumerate}

    We set $E:=r^{-1}(\mathbb{R}_+)$ and $S:=r^{-1}(1)$. 
\end{defi}

Let $g$ be a Riemannian metric on $M$ and $r\in C(M; [0, \infty))$ be an escape function on $M$. Then $M$ has a natural orthogonal decomposition into radial variable and angular variable: 

\begin{prop}\label{prop_radial_angular_decomposition}
    Let $r\in C(M; [0, \infty))$ be an escape function and $g$ be a Riemannian metric on $M$. We set $S=r^{-1}(1)$ and $E=r^{-1}((0, \infty))$ as in Definition \ref{defi_escape_function}. Then the vector field $\rmop{grad} r/|\rmop{grad} r|_g^2$ generates the flow $\{ \psi_t: E\to E\}_{t\geq 0}$ on $E$ with the following properties. 
    \begin{enumerate}
        \renewcommand{\labelenumi}{(\roman{enumi})}
        \item $r(\psi_t(x))=r(x)+t$ for $x\in E$. 
        \item The mapping
        \begin{equation}\label{eq_diffeo_end}
            \Psi: E \longrightarrow \mathbb{R}_+ \times S, \quad 
            \Psi (x):=(r(x), \psi_{r(x)-1}^{-1}(x))
        \end{equation}
        is a diffeomorphism with the inverse function 
        \[
            \Psi^{-1}(r, \theta)=\psi_{r-1}(\theta). 
        \]
        \item The decomposition $TE \simeq T\mathbb{R}_+ \oplus TS$ induced by \eqref{eq_diffeo_end} is orthogonal. 
    \end{enumerate}
\end{prop}

\begin{proof}
If we prove that the vector field generates the flow $\{ \psi_t: E\to E\}_{t\geq 0}$, then the properties from (i) to (iii) are proved easily. 

(i) The definition of $\psi_t$ implies 
\[
    \frac{\diff}{\diff t}(r(\psi_t(x)))=\jbracket{\diff r(\psi_t), \frac{\diff \psi_t}{\diff t}}=\frac{1}{|\grad r(\psi_t)|_g^2}\times \underbrace{\jbracket{\diff r(\psi_t), \grad r(\psi_t)}}_{=|\grad r(\psi_t)|_g^2}=1. 
\]
Thus 
\[
    r(\psi_t(x))=r(x)+\int_0^t \frac{\diff}{\diff t}(r(\psi_t(x)))\, \diff t=r(x)+t. 
\]

(ii) The smoothness and the form of inverse mapping are obvious from the definition of $\Psi: E\to \mathbb{R}_+\times S$. 

(iii) If $\gamma_\mathrm{rad}(t):=\Psi^{-1}(r+t, \theta)$ and $v_\mathrm{ang}\in T_\theta S$, then 
\[
    g\left( \frac{\diff \gamma_\mathrm{rad}}{\diff t}(0), v_\mathrm{ang}\right)
    =g\left( \frac{\grad r}{|\grad r|^2}, v_\mathrm{ang}\right)=0 
\]
by the fact that gradient vectors intersect level sets orthogonally. 

Thus the problem is that the integral curve $t\mapsto \psi_t(x)$ is defined for all $t\geq 0$. 

Fix $x\in E$ and consider the set 
    \[ B:=\left\{\, b\in (0, \infty) \,\middle|\, 
    \begin{aligned}
        &\exists \gamma \in C^\infty ([0, b]; E) \text{ s.t. } \gamma \text{ is an integral curve of } \\
        &\grad r/|\grad r|_g^2 \text{ with initial point } x
    \end{aligned}
    \,\right\}. \]
    If one prove 
    \begin{enumerate}
    \renewcommand{\labelenumi}{(\alph{enumi})}
    \item $B\neq \varnothing$, 
    \item that $B$ is an open subset of $(0, \infty)$ and  
    \item that $B$ is a closed subset of $(0, \infty)$, 
    \end{enumerate}
    then $B=(0, \infty)$ by the connectedness of $(0, \infty)$. 

(a) By the existence of the solutions to ordinary differential equations and $\diff r\neq 0$ near $x$, there exists an integral curve $\gamma: [-\varepsilon, \varepsilon]\to E$ of $\grad r/|\grad r|_g^2$ with initial point $x$. Thus $B\neq \varnothing$. 

(b) Let $b\in B$. Then there exists an integral curve $\gamma: [0, b]\to E$ of $\grad r/|\grad r|_g^2$ with the initial point $x$. Since $\gamma(b)\in E$, the vector field $\grad r/|\grad r|_g^2$ can be defined near $\gamma(b)$. Thus there exists an integral curve $\beta: [-\varepsilon, \varepsilon]\to M$ ($0<\varepsilon \ll 1$) of $\grad r/|\grad r|_g^2$ with the initial point $\gamma(b)$. Since $\gamma (t)=\beta(t-b)$ ($b-\varepsilon<t\leq b$) by the uniqueness of solutions to ordinary differential equations, we can extend $\gamma$ to  
\[ \Gamma (t):=
\begin{cases}
    \gamma(t) & \text{if } 0\leq t \leq b, \\
    \beta(t-b) & \text{if } b<t<b+\varepsilon. 
\end{cases}\]
This $\Gamma$ is an integral curve of $\grad r/|\grad r|_g^2$ defined for $t\in [0, b+\varepsilon]$ with the initial point $x$. Hence $(b-\varepsilon, b+\varepsilon)\subset B$. 

(c) It is enough to prove that if $\gamma: [0, b)\to E$ is an integral curve of $\grad r/|\grad r|_g^2$, then the limit $\lim_{t\to b-0}\gamma (t)\in E$ exists. We denote by $d(x, y)$ the distance associated with the Riemannian metric $g$. Since 
    \[ d(\gamma(s), \gamma(t))\leq \int_s^t \left| \frac{\diff \gamma}{\diff \tau}(\tau)\right|_g\, \diff \tau \leq |t-s| \underbrace{\max_{y\in r^{-1}([0, 1+b])} |X(y)|_g}_{\substack{\text{exists by compactness of} \\ r^{-1}([0, b+1])}}
    \]
for $0\leq s\leq t<b$ by the definition of the distance, the compactness of $r^{-1}([0, 1+b])$ implies the existence of the limit $\gamma(b):=\lim_{t\to b-0}\gamma (t)$. We have $\gamma (b)\in E$ since 
\[
    r(\gamma (t))=r(x)+\int_0^t \jbracket{\diff r, \frac{\grad r}{|\grad r|_g^2}}\, \diff t=r(x)+t\geq r(x)
\]
for $0\leq t<b$ and thus
\[
    r(\gamma (b))=\lim_{t\to b-0} r(\gamma (t))\geq r(x)>0. 
\]
Hence $b\in B$. 
\end{proof}

\section*{Acknowledgements}
The author thanks Professor Kenichi Ito, Professor Shu Nakamura and Professor Kouichi Taira for valuable discussion and advice. 

\bibliography{propagation_of_singularities_bib}

\begin{thebibliography}{10}

\bibitem{Bouclet11-1}
J.-M. Bouclet.
\newblock Semi-classical functional calculus on manifolds with ends and
  weighted {$L^p$} estimates.
\newblock {\em Ann. Inst. Fourier (Grenoble)}, 61(3):1181--1223, 2011.

\bibitem{Bouclet11-2}
J.-M. Bouclet.
\newblock Strichartz estimates on asymptotically hyperbolic manifolds.
\newblock {\em Anal. PDE}, 4(1):1--84, 2011.

\bibitem{Cordero-Nicola-Rodino15}
E.~Cordero, F.~Nicola, and L.~Rodino.
\newblock Propagation of the {G}abor wave front set for {S}chr\"odinger
  equations with non-smooth potentials.
\newblock {\em Rev. Math. Phys.}, 27(1):1550001, 2015.

\bibitem{Craig-Kappeler-Strauss96}
W.~Craig, T.~Kappeler, and W.~Strauss.
\newblock Microlocal dispersive smoothing for the {S}chr\"odinger equation.
\newblock {\em Comm. Pure Appl. Math}, 48:769--860, 1996.

\bibitem{Doi96}
S.~Doi.
\newblock Smoothing effects of {S}chr\"{o}dinger evolution groups on
  {R}iemannian manifolds.
\newblock {\em Duke Math. J.}, 82(3):679--706, 1996.

\bibitem{Dyatlov-Zworski19}
S.~Dyatlov and M.~Zworski.
\newblock {\em Mathematical theory of scattering resonances}, volume 200 of
  {\em Graduate Studies in Mathematics}.
\newblock American Mathematical Society, Providence, RI, 2019.

\bibitem{Hassell-Wunsch05}
A.~Hassell and J.~Wunsch.
\newblock The {S}chr\"{o}dinger propagator for scattering metrics.
\newblock {\em Ann. of Math. (2)}, 162(1):487--523, 2005.

\bibitem{Ito06}
K.~Ito.
\newblock Propagation of singularities for {S}chr\"odinger equations on the
  {E}uclidean space with a scattering metric.
\newblock {\em Comm. Partial Differential Equations}, 31(12):1735--1777, 2006.

\bibitem{Ito-Nakamura09}
K.~Ito and S.~Nakamura.
\newblock Singularities of solutions to the {S}chr\"{o}dinger equation on
  scattering manifold.
\newblock {\em Amer. J. Math.}, 131(6):1835--1865, 2009.

\bibitem{Kato-Kobayashi-Ito15}
K.~Kato, M.~Kobayashi, and S.~Ito.
\newblock Wave front set defined by wave packet transform and its application.
\newblock In {\em Nonlinear dynamics in partial differential equations},
  volume~64 of {\em Adv. Stud. Pure Math.}, pages 417--425. Math. Soc. Japan,
  Tokyo, 2015.

\bibitem{Kato-Kobayashi-Ito17}
K.~Kato, M.~Kobayashi, and S.~Ito.
\newblock Remark on characterization of wave front set by wave packet
  transform.
\newblock {\em Osaka J. Math.}, 54(2):209--228, 2017.

\bibitem{Martinez02}
A.~Martinez.
\newblock {\em An Introduction to Semiclassical and Microlocal Analysis}.
\newblock Springer-Verlag Berlin, 2002.

\bibitem{Martinez-Nakamura-Sordoni09}
A.~Martinez, S.~Nakamura, and V.~Sordoni.
\newblock Analytic wave front set for solutions to {S}chr\"{o}dinger equations.
\newblock {\em Adv. Math.}, 222(4):1277--1307, 2009.

\bibitem{Melrose94}
R.~Melrose.
\newblock Spectral and scattering theory for the {L}aplacian on asymptotically
  {E}uclidian spaces.
\newblock In {\em Spectral and scattering theory ({S}anda, 1992)}, volume 161
  of {\em Lecture Notes in Pure and Appl. Math.}, pages 85--130. Dekker, New
  York, 1994.

\bibitem{Melrose-SaBaretto-Vasy14}
R.~Melrose, A.~S\'{a}~Barreto, and A.~Vasy.
\newblock Analytic continuation and semiclassical resolvent estimates on
  asymptotically hyperbolic spaces.
\newblock {\em Comm. Partial Differential Equations}, 39(3):452--511, 2014.

\bibitem{Mourre80}
E.~Mourre.
\newblock Absence of singular continuous spectrum for certain selfadjoint
  operators.
\newblock {\em Comm. Math. Phys.}, 78(3):391--408, 1980/81.

\bibitem{Nakamura05}
S.~Nakamura.
\newblock Propagation of the homogeneous wave front set for {S}chr\"odinger
  equations.
\newblock {\em Duke. Math. J.}, 126(2):349--367, 2005.

\bibitem{Nakamura09}
S.~Nakamura.
\newblock Semiclassical singularity propagation property for {S}chr\"odinger
  equations.
\newblock {\em J. Math. Soc. Japan}, 61(1):177--211, 2009.

\bibitem{Pravda-Starov-Rodino-Wahlberg18}
K.~Pravda-Starov, L.~Rodino, and P.~Wahlberg.
\newblock Propagation of {G}abor singularities for {S}chr\"{o}dinger equations
  with quadratic {H}amiltonians.
\newblock {\em Math. Nachr.}, 291(1):128--159, 2018.

\bibitem{Robbiano-Zuily99}
L.~Robbiano and C.~Zuily.
\newblock Microlocal analytic smoothing effect for the {S}chr\"{o}dinger
  equation.
\newblock {\em Duke Math. J.}, 100(1):93--129, 1999.

\bibitem{Robbiano-Zuily02}
L.~Robbiano and C.~Zuily.
\newblock Analytic theory for the quadratic scattering wave front set and
  application to the {S}chr\"{o}dinger equation.
\newblock {\em Ast\'{e}risque}, (283):vi+128, 2002.

\bibitem{SaBarreto05}
A.~S\'{a}~Barreto.
\newblock Radiation fields, scattering, and inverse scattering on
  asymptotically hyperbolic manifolds.
\newblock {\em Duke Math. J.}, 129(3):407--480, 2005.

\bibitem{Schulz-Wahlberg17}
R.~Schulz and P.~Wahlberg.
\newblock Equality of the homogeneous and the {G}abor wave front set.
\newblock {\em Comm. Partial Differential Equations}, 42(5):703--730, 2017.

\bibitem{Wunsch99}
J.~Wunsch.
\newblock Propagation of singularities and growth for {S}chr\"odinger
  operators.
\newblock {\em Duke Math. J.}, 98(1):137--186, 1999.

\bibitem{Zworski12}
M.~Zworski.
\newblock {\em Semiclassical analysis}, volume 138 of {\em Graduate Studies in
  Mathematics}.
\newblock American Mathematical Society, Providence, RI, 2012.

\end{thebibliography}
\bibliographystyle{plain}

\end{document}